\documentclass[a4paper,reqno]{amsart}

\usepackage[utf8]{inputenc}
\usepackage{amssymb}
\usepackage{enumitem}
\usepackage{mathrsfs}
\usepackage{mathtools}
\usepackage{amsmath}
\usepackage [dvipsnames] { xcolor }
\usepackage[colorlinks=true]{hyperref}

\hypersetup{linkcolor=Blue,urlcolor=cyan,citecolor=cyan}

\setlist[enumerate]{label=\emph{(\roman*)}}

\usepackage{environ}

\newtheorem{theorem}{Theorem}[section]
\newtheorem{corollary}[theorem]{Corollary}
\newtheorem{lemma}[theorem]{Lemma}
\newtheorem{proposition}[theorem]{Proposition}

\theoremstyle{definition}
\newtheorem{definition}[theorem]{Definition}
\newtheorem{remark}[theorem]{Remark}
\numberwithin{equation}{section}
\newcommand{\R}{\mathbb{R}}

\def \R {{\mathbb{R} }}
\def \d {{\rm{d}}}
\def \pt{\partial_{t}}

\begin{document}

\parindent=0pt

	\title[Observability for 1D Schr\"odinger equations]
	{Quantitative observability for one-dimensional Schr\"odinger equations with potentials}
	
	\author[P. Su]{Pei Su}
	\address{Department of Mathematical Analysis,
Faculty of Mathematics and Physics,
Charles University,
Sokolovská 83,
186 75 Praha 8, Czech Republic.
\newline
Present address: Laboratoire de Mathématiques d’Orsay (UMR 8628), Université Paris-Saclay, CNRS, 91405 Orsay Cedex, France.
}
	\email{pei.su@universite-paris-saclay.fr}
	
	\author[C. Sun]{Chenmin Sun}
	\address{CNRS, Universit\'e Paris-Est Cr\'eteil, Laboratoire d'Analyse et de Math\'ematiques appliqu\'ees, UMR  8050 du CNRS, 
	94010 Cr\'eteil cedex, France.}
	\email{chenmin.sun@cnrs.fr}
	
	\author[X.~Yuan]{Xu Yuan}
	\address{Academy of Mathematics and Systems Science, Chinese Academy of Sciences, Beijing 100190, P.R. China.}
	\email{xu.yuan@amss.ac.cn} 
	
	\begin{abstract}
  In this note, we prove the quantitative observability with an explicit control cost for the 1D Schr\"odinger equation over $\R$ with real-valued, bounded continuous potential on thick sets.  Our proof relies on different techniques for low-frequency and high-frequency estimates. In particular, we extend the large time observability result for the 1D free Schr\"odinger equation in Theorem 1.1 of Huang-Wang-Wang \cite{HWW} to any short time. 
	As another byproduct, we extend the spectral inequality of Lebeau-Moyano~\cite{LeM} for real-analytic potentials to bounded continuous potentials in the one-dimensional case. 
 \end{abstract}
	\maketitle
\section{Introduction}
\subsection{Main result}
Consider the 1D Schr\"odinger equation, 
   \begin{equation}\label{equ:LS}
   \begin{aligned}
   i\pt u-\partial_{x}^{2}u+V(x)u=0,\quad u_{|t=0}=u_{0}\in L^{2}(\R),
   \end{aligned}
   \end{equation}
  where the potential $V$ is a real-valued, continuous and bounded function. We are interested in the observability of the Schr\"odinger equation \eqref{equ:LS} on rough sets, which concerns the following inequality,
\begin{equation}\label{def:ob}
	\|u_0\|_{L^2(\R)}^2 \leq C(T,V,\Omega) \int_0^T \|u(t)\|_{L^2(\Omega)}^2 \d t,
\end{equation}
for all solutions $u(t)$ of~\eqref{equ:LS} where $\Omega \subset \R$ is a measurable subset.

\smallskip
The inequality \eqref{def:ob} measures how solutions of Schr\"odinger equations can concentrate on subsets of the domain. Such a property is linked to the high-frequency wave propagation phenomenon and to concentration properties for quasimodes of the Schr\"odinger operator. The results are sensitive for different underlying manifolds and the corresponding Schr\"odinger operators. Another motivation for establishing the observability estimate \eqref{def:ob} is to prove the exact controllability for the associated control system. See Corollary \ref{cor:control} for the precise statement.

\smallskip
In the general framework, there are three parameters that affect the observability estimates for Schr\"odinger type equations. These are the underlying geometry (the background manifold on which the equation is posed and the associated Schr\"odinger operator), the control region $\Omega$, and the time $T>0$ to achieve the observability. When observability holds for any time $T>0$, the control cost, i.e. the blow-up rate of the optimal constant $C(T,V,\Omega)$ is also an object of study. 

\smallskip
In this note, we address the observability problem for the 1D Schr\"odinger equation on an unbounded set by measurable control regions. To the best of our knowledge, this setup is much less studied in the literature. To state the main result, we recall the thickness condition for the control region.

\begin{definition}\label{def:thick}
   Let $0<\zeta<1$ and $0<L<\infty$. We say that $\Omega\subset \R$ is a $(L,\zeta)$-thick set if $\Omega$ is a  measurable set and satisfies   
   \begin{equation*}
       \left|\Omega\cap [x,x+L]\right|\ge \zeta L,\quad \mbox{for any} \  x\in \R.
  \end{equation*}
   \end{definition}

   The main result of this note is the following quantitative observability for~\eqref{equ:LS}.
   
  \begin{theorem}\label{thm:main1}
  Let $\Omega$ be a $(L,\zeta)$-thick set and let $V\in C(\R)\cap L^{\infty}(\R)$. Then there exists a constant $C=C(V,L,\zeta)>0$, depending only on $V$, $L$ and $\zeta$, such that for any $T>0$ and any solution $u$ of~\eqref{equ:LS} we have 
  \begin{equation}\label{est:Obser}
  \|u_{0}\|_{L^{2}(\R)}^{2}\le C e^{\frac{C}{T^{2}}}
   \int_{0}^{T}\left\|u(t)\right\|^{2}_{L^{2}(\Omega)}\d t.
  \end{equation}
  \end{theorem}

  \begin{remark}
	In \cite[Theorem 1.1]{HWW}, the authors considered the 1D free Schr\"odinger equation. They showed that \eqref{def:ob} holding for some time $T>0$ is equivalent to the control region $\Omega$ being thick. Theorem \ref{thm:main1} extends their results to any observability time $T>0$ for general 1D Schr\"odinger operators with bounded and continuous potential. In particular, we answer the question raised in Remark (a1) below \cite[Theorem 1.1]{HWW} concerning the short time observability for the 1D Schr\"odinger equation. Moreover, the proof of Theorem \ref{thm:main1} is purely quantitative which provides the upper bound for the control cost in terms of $T>0$ and the parameter defining the thickness. See Section~\ref{SS:MAIN} for more precise comments on the proof.
\end{remark} 

As a consequence of the classical Hilbert uniqueness method (see \cite{Li}), we have the following exact controllability result for 1D Schr\"odinger equations.
\begin{corollary}\label{cor:control}
Let $\Omega$ be a $(L,\zeta)$-thick set and let $V\in C(\R)\cap L^{\infty}(\R)$.	For any $T>0$ and any $(u_0,u_1)\in L^2(\R)\times L^{2}(\R)$, there exists a control $f\in L^2((0,T)\times\R)$ such that the unique solution of the 1D inhomogeneous Schr\"odinger equation
\begin{equation*}
 i\partial_t u-\partial_x^2 u+V(x)u={\mathbf{1}}_{\Omega}f
\end{equation*}
	with initial data $u(0)=u_{0}$ satisfies $u(T)=u_1$.
\end{corollary}

\begin{remark}\label{remark:intro}
We mention here that, it is easy to extend Theorem~\ref{thm:main1} and Corollary~\ref{cor:control} to the case of the potential $V\in L^{\infty}(\R)$, based on a standard density argument. However, to simplify the notation and to make our analysis more concise, we only consider the case of the potential $V\in C(\R)\cap L^{\infty}(\R)$ in this note (see Remark~\ref{remark:infty} for more details of the proof for this extension).
\end{remark}

\subsection{Previous results}
In the existing literature, observability for Schr\"odinger equations on compact manifolds and bounded domains has been extensively studied. Sufficient geometric conditions are often imposed on the control region $\Omega$. When $\Omega$ is an open set satisfying the geometric control condition (GCC), Lebeau \cite{Le} proved that observability is true for an arbitrarily short time \(T>0\). The essential behind  Lebeau's theorem is the infinite speed of propagation of singularities for high-frequency wave-packets of solutions to the Schr\"odinger equation.

\smallskip
The condition GCC is in general not necessary. The stability (or instability) property of the geodesic flow of the underlying manifold plays a decisive role. Specifically, it has been shown (though this list is not exhaustive) that observability is valid for any \(T>0\) and for any non-empty open control region if the underlying manifold is a torus \cite{AM14,BBZ,BZ4,Ja}, or compact hyperbolic surfaces \cite{Ji} (see also \cite{AR} for negatively curved manifolds), or the disk \cite{ALM14}, provided the control region encompasses a neighborhood of some portion of the boundary. We also mention here that for certain subelliptic Schr\"odinger equations, the minimal observability time is strictly positive (see \cite{BuSun,FeLe}).

\smallskip
In recent years, there has been a growing focus on controllability and observability for Schr\"odinger equations by rough control regions or in non-compact settings (see for example \cite{BZ5,HWW,MLe,Pr,Taufer,WWZ}). 
For the Schr\"odinger equation on two-dimensional tori, \cite{BZ5} showed that the short time observability result holds for any measurable control region with a positive measure. Very recently, this result has been extended to Schr\"odinger equations (with periodic potential) on \(\R^2\) having a periodic control region. It is worth mentioning that the thickness condition (see Definition \ref{def:thick}) in the one-dimensional context is a version of GCC on rough sets. The interest of the result in \cite{MLe} lies in the fact that the control region, though \(2\pi \mathbb{Z}^2\)-periodic, might not satisfy  GCC. Lastly, we refer to \cite{HWW,Pr} for discussions on the observability of Schr\"odinger operators on $\R^{d}$ with confining potentials.

\subsection{Comments on the proof}

Our proof of Theorem \ref{thm:main1} relies on three steps:
spectral inequality for low-frequency part, high-frequency part observability via resolvent estimate (of Hautus type inequality), and a quantitative gluing argument. 

\smallskip
The first part of our analysis is the proof of the low-frequency observability by establishing the spectral inequality for 1D Schr\"odinger operators with continuous bounded potentials (see Lemma \ref{le:spectra}). This particularly generalizes the spectral inequality in \cite{LeM} for real-analytic potentials to bounded continuous potentials in the one-dimensional context. Our proof of the spectral inequality leverages a very recent work \cite{ZHU} on the propagation of smallness for elliptic equations on the plane.

\smallskip
Second, the resolvent estimate for the high-frequency part is essentially not new and has already been obtained in \cite[Proposition 1]{Gr} for the 1D free Schr\"odinger operator (i.e. the case of~$V=0$). The proof in \cite{Gr} utilizes the Logvinenko-Sereda uncertainty principle (see \cite[Theorem 2]{Ko}), which exploits the straightforward geometric feature of the 1D Fourier space. More precisely, in the 1D case, the measure of $\xi$ such that $||\xi|-\lambda|\leq O(1)$ is bounded by $O(1)$, independent of the position of $\lambda$. Instead, we provide a different, more elementary proof (see Lemma \ref{le:resolvent}) of the high-frequency resolvent estimate for 1D Schr\"odinger operators.

\smallskip
Another delicate issue in the analysis is gluing low-frequency and high-frequency estimates to obtain the full observability estimate. We briefly revisit the classical compactness-uniqueness method that is extensively used in the literature in the compact setting (in the context where the Schr\"odinger operator has a compact resolvent). The strategy is due to Bardos-Lebeau-Rauch \cite{BLR} (see also \cite[Chapter 6]{TuWe} for a similar abstract framework). 
Given a general Schr\"odinger type equation associated with a self-adjoint operator $\mathcal{A}$ on a Hilbert space $\mathcal{H}$ and a control operator $\mathcal{C}$:
$$ i\partial_tu=\mathcal{A}u,\quad u_{|t=0}=u_0\in \mathcal{H}. $$
We are concerned with the following observability inequality:
\begin{align}\label{ob:abstract}
	\|u_0\|_{\mathcal{H}}^2\leq C\int_0^T\|\mathcal{C}u(t)\|_{\mathcal{H}}^2\d t.
\end{align}
The compactness-uniqueness strategy can be summarized as follows. Let us first assume that we have already established the high-frequency observability estimate,
\begin{align}\label{ob:highabstract}
	\|u_0\|_{\mathcal{H}}^2\leq C\int_0^T\|\mathcal{C}u(t)\|_{\mathcal{H}}^2\d t+C\|\mathcal{A}^{-1}u_0\|_{\mathcal{H}}^2.
\end{align}
Here, for the sake of clarity, we assume that $0\notin \mathrm{Spec}(\mathcal{A})$.
Given $T>0$, we define
\[ \mathcal{N}_T:=\{u_0\in \mathcal{H}:\; \mathcal{C}u(t)\equiv 0 \text{ in } L^2((0,T);\mathcal{H})\}. \]
When $\mathcal{A}$ has a compact resolvent (thus $\mathrm{Spec}(\mathcal{A})$ is discrete), the compact embedding and \eqref{ob:highabstract} ensures that $\mathcal{N}_T$ exists as a finite-dimensional linear subspace of $\mathcal{H}$. Furthermore, it can be deduced that the restriction of $\mathcal{A}$ on $\mathcal{N}_T$ is a well-defined linear operator. By taking an eigenfunction $\phi \neq 0$ of $\mathcal{A}$ on $\mathcal{N}_T$, we find that $\mathcal{C}\phi\equiv 0$. Thus, the observability \eqref{ob:abstract} results from \eqref{ob:highabstract} coupled with the unique continuation property of eigenfunctions:
\[ \mathcal{A}\phi=\lambda \phi\ \ \mbox{and}\ \ \mathcal{C}\phi\equiv 0 \Longrightarrow \phi\equiv 0. \]

\smallskip
However, when considering the non-compact setting where $\mathcal{A}$ does not have a compact resolvent, we are not aware of any compactness-uniqueness type argument to deduce \eqref{ob:abstract} from \eqref{ob:highabstract}. Actually, in \cite{MLe}, this issue does not exist as the problem can be reduced to the compact setting $\mathbb{T}^2$, thanks to the periodicity of both the control region and the potential. Here, we prove a quantitative unique continuation property for Schr\"odinger equations, which facilitates to glue the high-frequency observability with low-frequency estimates. This idea was introduced in Phung \cite{Phung01} for the Schr\"odinger equation on bounded domain over $\R^{d}$.  Following this strategy, we begin by establishing the backward-in-time observability for the associated heat semigroup (see Proposition \ref{prop:heat}). This allows us to deduce a quantitative unique continuation estimate for Schr\"odinger equations by taking the Fourier–Bros–Iagolnitzer (FBI) transformation in time. 
We refer to \cite[Appendix A]{BBZ} where another quantitative uniqueness-compactness argument was introduced in the compact setting.

\smallskip
Though the passage from the observability of heat-semigroup and the high-frequency observability to the full observability is quite robust in both the compact and non-compact settings, 
it is worth mentioning that this zigzag path exceeds what is essentially required to ensure the Schr\"odinger observability estimate.  For instance, we consider the half Laplacian $\mathcal{A}=|\nabla|$ on the 1D torus $\mathbb{T}$. It is a known fact that for any non-empty open set $\omega\subset \mathbb{T}$, the associated heat semigroup is not observable by control operator $\mathcal{C}=\mathbf{1}_{\omega}$ (see \cite[Subsection 2.1]{Koe}). However, the Schr\"odinger semigroup $e^{\pm it\mathcal{A}}$ (which is the half-wave operator) is observable by control operator $\mathcal{C}=\mathbf{1}_{\omega}$ for a certain $T>0$ (this is a very special case of the exact controllability for the wave equation under the geometric control condition). This discussion raises a natural question of extending the compactness-uniqueness approach, especially concerning the observability estimates of Schr\"odinger type equations in non-compact domains. We expect that such an extension would be potentially useful when addressing observability estimates for many other Schr\"odinger type equations on non-compact domains.

\smallskip
The note is organized as follows: First, in Section~\ref{SS:Nota}, we state notation and conventions. Second, in Section~\ref{SS:SPEC}, we derive the propagation of smallness for the 2D elliptic equation and then establish the spectral inequality for the Schr\"odinger operator $H$ as a consequence. Third, in Section~\ref{SS:RESO}, we deduce the high-frequency observability based on the resolvent estimate. Then, in Section~\ref{SS:MAIN}, we complete the proof of Theorem~\ref{thm:main1} based on the gluing of the low-frequency and high-frequency estimates. Last, for self-contained reason, we give the details of proof for Theorem~\ref{thm:ZHU} and Proposition~\ref{prop:heat} in Appendix~\ref{AppA} and~\ref{AppB}, respectively.

\subsection*{Acknowledgments}
The author P.~S. is supported by the ERC-CZ Grant CONTACT LL2105 funded by the Ministry of Education, Youth and Sport of the Czech Republic.
The author C.~S. is partially supported by the PEPS-JCJC and ANR project SmoothANR-22CE40-0017. The author X.~Y. is partially supported by the Direct Grant for Research (Project No. 4053575) from CUHK.
The author X.~Y. would like to thank Gengsheng Wang and Yubiao Zhang for early discussions on the observability of Schr\"odinger equations. The authors would like to thank Yunlei Wang for discussion about the regularity condition of the potential. The authors are also grateful to the anonymous referee for careful reading and useful suggestions, which led to an improved version of this note.

\subsection*{Additional acknowledgments}
The authors would like to warmly thank

\section{Notation and conventions}\label{SS:Nota}
The Fourier transform of a function $f\in L^{1}(\R)$, denoted by $\widehat{f}$, is defined as:
\begin{equation*}
    \widehat{f}(\xi)=\frac{1}{\sqrt{2\pi}}\int_{\R}e^{- i x\xi}f(x)\d x,\quad \mbox{for any}\ \xi\in \R.
\end{equation*}
Recall that, the Fourier transform defines a linear bounded operator from $L^{1}(\R)\cap L^{2}(\R)$ to $L^{2}(\R)$. Moreover,  this operator is an isometry and so  there is a unique bounded extension $\mathcal{F}$ defined in all $L^{2}(\R)$. The operator $\mathcal{F}$ is called the Fourier transform in $L^{2}(\R)$.
To shorten notation, we denote $\widehat{f}=\mathcal{F}f$ for $f\in L^{2}(\R)$.

\smallskip
Note that, if $u(t)$ is a solution for~\eqref{equ:LS},
then for any $\theta>0$, $U(t)=e^{i\theta t}u(t)$ is a solution for
\begin{equation*}
i\pt U-\partial_{x}^{2}U+V(x)U+\theta U=0,\quad U_{|t=0}=u_{0}\in L^{2}(\R).
\end{equation*}
This gauge transformation leaves the observability inequality \eqref{est:Obser} invariant.
Therefore, without loss of generality, we can assume the potential $V$ is a continuous bounded function with $V\ge 1$ throughout the article. To shorten notation, for a given potential $V\in C(\R)\cap L^{\infty}(\R)$ with $V\ge 1$, we denote
\begin{equation*}
\|V\|_{\infty}=\|V\|_{L^{\infty}(\R)}\quad \mbox{and}\quad {H}=-\partial_{x}^{2}+V\quad \mbox{with domain}\ H^{2}(\R).
\end{equation*}
Note also that, the operator ${H}$ on $L^{2}(\R)$ with domain $H^{2}(\R)$ is self-adjoint with non-negative continuous spectrum. Based on the spectral theorem (see for instance~\cite[Section 2.5]{DAVIES}), there exists a spectral measure $\d m_{\lambda}$ such that 
\begin{equation*}
F(\sqrt{{H}})=\int_{0}^{\infty}F(\lambda)\d m_{\lambda},\quad \mbox{for any bounded function} \ F.
\end{equation*}
Moreover, for any bounded functions $F$ and $G$, we have 
\begin{equation*}
\left(F(\sqrt{{H}})f,G(\sqrt{{H}})f\right)_{L^{2}(\R)}=\int_{0}^{\infty}F(\lambda)\overline{G(\lambda)}(\d m_{\lambda}f,f)_{L^{2}(\R)},\quad \mbox{for}\ f\in L^{2}(\R).
\end{equation*}
Here, $\d m_{\lambda}$ is the spectral measure of the operator $\sqrt{{H}}$.
In particular, for $\mu>0$, the spectral projector $\Pi_{\mu}$, associate with the function $F(\lambda)=\textbf{1}_{\{\lambda\le \mu\}}$, is defined by 
\begin{equation*}
\Pi_{\mu}=\textbf{1}_{\left\{\sqrt{{H}}\le \mu\right\}}=\int_{0}^{\mu}\d m_{\lambda}.
\end{equation*} 
On the other hand, the operator $i{H}$ generates a unitary group $e^{it{H}}$ and so the solution for~\eqref{equ:LS} can be written as $u(t)=e^{it{H}}u_{0}\in L^{2}(\R)$. 

\smallskip
For future reference, we define
\begin{equation*}
D_{1}=\R\times \left[-\frac{1}{2},\frac{1}{2}\right],\quad 
D_{2}=\R\times \left[-\frac{3}{2},\frac{3}{2}\right]
\quad \mbox{and}\ \ D_{3}=\R\times \left[-\frac{5}{2},\frac{5}{2}\right].
\end{equation*}

Next, for any $\ell\in \mathbb{Z}$, we define
\begin{equation*}
I_{1\ell}=[\ell,\ell+1],\quad I_{2\ell}=[\ell-1,\ell+2]\quad \mbox{and}\ \ I_{3\ell}=[\ell-2,\ell+3].
\end{equation*}
Moreover, for any $\ell\in \mathbb{Z}$, we set
\begin{equation*}
D_{1\ell}=I_{1\ell}\times \left[-\frac{1}{2},\frac{1}{2}\right],\quad 
D_{2\ell}=I_{2\ell}\times \left[-\frac{3}{2},\frac{3}{2}\right]
\quad \mbox{and}\ \ D_{3\ell}=I_{3\ell}\times \left[-\frac{5}{2},\frac{5}{2}\right].
\end{equation*}

\smallskip
Without loss of generality, we can assume for a given thick set $\Omega\subset \R$, there exists a constant $0<\zeta< 1$ such that
 \begin{equation*}
       \left|\Omega\cap [x,x+1]\right|\ge \zeta \quad \mbox{for any}\ x\in \R,
   \end{equation*}
   that is, the constant $L=1$ in the Definition~\ref{def:thick}. For a given thick set $\Omega$, we set
\begin{equation*}
    \Omega_{\ell}=\Omega \cap I_{1\ell}=\Omega\cap [\ell,\ell+1],\ \ \mbox{for any}\ \ell\in \mathbb{Z}.
\end{equation*}

By abuse of notation, in this article, we use the same letters $\alpha$ and $C$ for some small or large constants and state their dependency on other parameters.

\smallskip
For $a\in \R^{2}$ and $r>0$, we denote by $B_{r}(a)$ the ball of $\R^{2}$ of center $a$ and of radius $r$.
To simplify notation, we also denote by $B_{r}$ the ball in $\R^{2}$ centred at the origin with radius $r>0$. 

\smallskip
For $\delta>0$, we denote by $\mathcal{H}_{\delta}$ the $\delta$-dimensional Hausdorff content, that is, for a subset $E\subset \mathbb{R}^{2}$, we define
\begin{equation*}
    \mathcal{H}_{\delta}(E)=\inf\left\{\sum_{n=1}^{\infty}r_{n}^{\delta}:E\subset \bigcup_{n=1}^{\infty}B_{r_{n}}(a_{n}),\ a_{n}\in \R^{2}\right\}.
    \end{equation*} 
Let $\omega\subset B_{1}\cap{\ell_{0}}$ satisfy $|\omega|>0$ for some line $\ell_{0}$ in $\R^{2}$. From the definition of 1D Lebesgue measure and $\delta$-Hausdorff content, we have 
\begin{equation}\label{equ:measu}
    |\omega|=\inf\left\{\sum_{n=1}^{\infty}|I_{n}|:\omega\subset\bigcup_{n=1}^{\infty}I_{n},\ I_{n}\subset \ell_{0} \right\}=2\mathcal{H}_{1}(\omega).
\end{equation}

\section{Spectral inequality}\label{SS:SPEC}

\subsection{Propagation of smallness}
In this subsection, we introduce the propagation of smallness for solutions of elliptic equations in $\R^{2}$. We start with the following technical lemma for the second-order ODE
\begin{equation}\label{equ:ode}
-\varphi''(x)+V(x)\varphi(x)=0.
\end{equation}
\begin{lemma}\label{le:ODE}
Let $I=[a,b]$ be a finite interval and let $V\in C(\R)\cap L^{\infty}(\R)$ with $V\ge 1$. There exists a $C^{2}$ positive solution $\varphi$ of~\eqref{equ:ode} on the interval $I$ such that 
\begin{equation*}
1\le \varphi(x)\le e^{(b-a)^{2}\|V\|_{{\infty}}},\quad \mbox{for any}\ x\in I.
\end{equation*}
\end{lemma}

\begin{remark}
The continuity of the potential $V$ is to ensure the existence and regularity of the solution for the second-order ODE~\eqref{equ:ode}. Such a regular solution can help us to reduce the elliptic equation~\eqref{equ:ellV} of non-divergence form to divergence form~\eqref{equ:2Dell}. This is the only reason why we assume $V\in C(\R)$. Actually, we expect Theorem~\ref{thm:main1} still hold true for any potential $V\in L^{\infty}(\R)$.
\end{remark}

\begin{proof}[Proof of Lemma~\ref{le:ODE}]
Consider the following initial-value problem 
\begin{equation}\label{equ:ODEcau}
-\varphi''(x)+V(x)\varphi(x)=0\quad \mbox{with}\ \ 
(\varphi(a),\varphi'(a))=(1,0).
\end{equation}
Based on Picard's existence theorem and $V\in C(\R)\cap L^{\infty}(\R)$, there exists a unique $C^{2}$ solution $\varphi$ to the initial-value problem~\eqref{equ:ODEcau}. We now establish the lower and upper bounded estimates for $\varphi$ on the interval $I$. First, from~\eqref{equ:ODEcau} and $V\in C(\R)\cap L^{\infty}(\R)$ with $V\ge 1$, we see that $\varphi(x)\ge 0$ on $I$. Therefore, for any $x\in I$, we have 
\begin{equation*}
\varphi'(x)=\int_{a}^{x}V(s)\varphi(s)\d s\ge 0\Longrightarrow
\varphi(x)\ge \varphi(a)=1.
\end{equation*}
Second, using again~\eqref{equ:ODEcau} and $V\in C(\R)\cap L^{\infty}(\R)$ with $V\ge 1$,
\begin{equation*}
\varphi(x)\le 1+ (b-a)\|V\|_{\infty}\int_{a}^{x}\varphi(s)\d s,\quad \mbox{for any}\ x\in I.
\end{equation*}
It follows from Grönwall's inequality that
\begin{equation*}
\varphi(x)\le e^{\|V\|_{\infty}\int_{a}^{x}(b-a)\d s}\le e^{(b-a)^{2}\|V\|_{{\infty}}},\quad \mbox{for any}\ x\in I.
\end{equation*}
Combining the above inequalities, we complete the proof of Lemma~\ref{le:ODE}.
\end{proof}

Second, we consider the 2D elliptic equation in divergence form
\begin{equation}\label{equ:2Dell}
\nabla \cdot \left(A(z)\nabla \phi(z)\right)=0\quad \mbox{in} \ B_{4}.
\end{equation}
Here $z=(x,y)\in \R^{2}$, and the real symmetric matrix $A(z)=(a_{jk}(z))_{2\times2}$ is elliptic, that is, there exists a constant $\Lambda>1$ such that  

\begin{equation}\label{est:ell}
\Lambda^{-1}\le \xi^{T}A(z)\xi\le \Lambda,\quad \mbox{for any}\ \xi\in B_{1} \ \mbox{and}\ z\in B_{4}.
\end{equation}

We recall the following propagation of smallness for solutions to~\eqref{equ:2Dell} from~\cite{ZHU}.
\begin{theorem}[\cite{ZHU}]\label{thm:ZHU}
Let $\omega\subset B_{1}\cap \ell_{0}$ satisfy $|\omega|>0$ for some line $\ell_{0}$ in $\R^{2}$ with the normal vector $\boldsymbol{\rm{e}}_{0}$. There exist some constants $\alpha=\alpha(\Lambda, |\omega|)\in(0,1)$ and $C=C(\Lambda,|\omega|)>0$, depending only on $\Lambda$ and $|\omega|$, such that for any real-valued $H_{loc}^{2}$ solution $\phi$ of~\eqref{equ:2Dell} with $A\nabla \phi\cdot \boldsymbol{\rm{e}}_{0}=0$ on $B_{1}\cap \ell_{0}$, we have 
\begin{equation}\label{eq:smallness}
\sup_{B_{1}}|\phi|\le C\left(\sup_{\omega}|\phi|^{\alpha}\right)\left(\sup_{B_{2}}|\phi|^{1-\alpha}\right).
\end{equation}
\end{theorem}
\begin{remark}\label{rk:technical}
By scaling and translation, the interpolation inequality \eqref{eq:smallness} remains true if we replace $B_1,B_2$ by balls $B_r(a),B_{2r}(a)$ for $H_{loc}^2$ solutions to \eqref{equ:2Dell} in $B_{4r}(a)$, where $a \in \R^2$ and $r>0$, and the constant $C$ depends only on $\Lambda,|\omega|$ and $r>0$.
\end{remark}
\begin{remark}
We mention here that Theorem~\ref{thm:ZHU} is just a special version of~\cite[Theorem 1.1]{ZHU}. Actually,~\cite[Theorem 1.1]{ZHU} shows the propagation of smallness for solutions from any $\omega\subset B_{1}$ lying on a line with $\mathcal{H}_{\delta}(\omega)>0$. Here $\mathcal{H}_{\delta}(\omega)$ means $\delta$-dimensional Hausdorff content of $\omega$. For the sake of completeness and the readers' convenience, the sketch of the proof for Theorem~\ref{thm:ZHU} is given in Appendix~\ref{AppA}.
\end{remark} 

Last, we introduce the $L^{2}$-propagation of smallness for $H_{\emph{loc}}^{2}$ solution of the following 2D elliptic equation in nondivergence form
\begin{equation}\label{equ:ellV}
-\Delta \phi(z)+V(x)\phi (z)=0\quad \mbox{with}\ \ \partial_{y}\phi_{|y=0}=0.
\end{equation}
Following~\cite[Section 2]{LOGUNOV}, we see that the equation~\eqref{equ:ellV} in nondivergence form can be reduced to the divergence form~\eqref{equ:2Dell}. More precisely, for the given potential $V\in C(\R)\cap L^{\infty}(\R)$ with $V\ge 1$, we first consider the positive $C^{2}$ solution $\varphi$ constructed in Lemma~\ref{le:ODE} for the second-order ODE~\eqref{equ:ode}.
Then, by an elementary computation, on $[a,b]\times \R$, we deduce that 
\begin{equation}\label{equ:reduc}
\left\{\begin{aligned}
\partial_{y}\phi_{|y=0}=0
&\Longrightarrow \partial_{y}\left(\frac{\phi(z)}{\varphi(x)}\right)_{|y=0}=0,\\
-\Delta \phi(z)+V(x)\phi(z)=0&\Longrightarrow \nabla \cdot \left(\varphi^{2}(x)\nabla\left(\frac{\phi(z)}{\varphi(x)}\right) \right)=0.
\end{aligned}\right.
\end{equation}

Combining the above reduction with Theorem~\ref{thm:ZHU}, we now establish the following $L^{2}$-propagation of smallness for the $H_{\emph{loc}}^{2}$ solution of 2D elliptic equation~\eqref{equ:ellV}. The proof is inspired by the techniques developed in Burq-Moyano~\cite[Section 2]{BurqMoyano}.
\begin{proposition}\label{pro:L2}
Let $\ell\in \mathbb{Z}$ and let $V\in C(\R)\cap L^{\infty}(\R)$ with $V\ge 1$. Then for any measurable set $\omega\subset I_{1\ell}$ with $|\omega|>0$ and any real-valued $H_{loc}^{2}$ solution $\phi$ of~\eqref{equ:ellV}, we have 
\begin{equation*}
\|\phi\|_{L^{2}(D_{1\ell})}\le C\|\phi\|_{L^{2}(\omega)}^{\alpha}\left(\sup_{D_{2\ell}}|\phi|^{1-\alpha}\right),
\end{equation*}
where $\alpha=\alpha(V,|\omega|)\in (0,1)$ and $C=C(V,|\omega|)>0$ depend only on $V$ and $|\omega|$.
\end{proposition}

\begin{remark}
In the spirit of Burq-Moyano~\cite{BurqMoyano} (see also~\cite{BurqMoyanoJEMS} in a similar context), the $L^{2}$-propagation of smallness can be used to establish the spectral inequality which can help us to obtain the observability estimate for 1D homogenous heat equation with a potential (see more details in \S\ref{SS:LOW}).
\end{remark}

\begin{proof}[Proof of Proposition~\ref{pro:L2}]

\textbf{Step 1.} $L^{\infty}$-propagation of smallness. We claim that, for any measurable set $\omega\subset I_{1\ell}$ with $|\omega|>0$ and any $H_{loc}^{2}$ solution $\phi$ of~\eqref{equ:ellV},
\begin{equation}\label{est:Linfty}
\sup_{D_{1\ell}}|\phi|\le C\left(\sup_{\omega}|\phi|^{\alpha}\right)\left(\sup_{D_{2\ell}}|\phi|^{1-\alpha}\right),
\end{equation}
where $\alpha=\alpha(V,|\omega|)\in (0,1)$ and $C=C(V,|\omega|)>0$ depend only on $V$ and $|\omega|$.
Indeed, we first consider the positive $C^{2}$ solution $\varphi$ constructed in Lemma~\ref{le:ODE} for the second-order ODE~\eqref{equ:ode} on $[\ell-5,\ell+5]$.
Then, from the reduction~\eqref{equ:reduc}, we can apply Theorem~\ref{thm:ZHU} to $(\phi/\varphi)$ with $\ell_{0}=\left\{(x,y)\in \R^{2}:y=0\right\}$. Then, according to scaling and translation, there exist some constants $\alpha=\alpha(\varphi,|\omega|)\in (0,1)$ and $C=C(\varphi,|\omega|)>0$, depending only on $|\omega|$ and the lower and upper bounds of $\varphi$ on $[\ell-5,\ell+5]$, such that for any $H_{loc}^{2}$ solution $\phi$ of~\eqref{equ:ellV}, we have 
\begin{equation*}
\sup_{B_{1\ell}}\left|\frac{\phi}{\varphi}\right|\le C\left(\sup_{\omega}\left|\frac{\phi}{\varphi}\right|^{\alpha}\right)\left(\sup_{B_{2\ell}}\left|\frac{\phi}{\varphi}\right|^{1-\alpha}\right),
\end{equation*}
where $B_{1\ell}$ and $B_{2\ell}$ are defined by 
\begin{equation*}
B_{1\ell}=B_{\frac{\sqrt{2}}{2}} \left(\left(\ell+\frac{1}{2},0\right)\right)\quad \mbox{and}\quad 
B_{2\ell}=B_{\sqrt{2}}\left(\left(\ell+\frac{1}{2},0\right)\right).
\end{equation*}
It follows directly from Lemma~\ref{le:ODE} that 
\begin{equation*}
\sup_{B_{1\ell}}\left|\frac{\phi}{\varphi}\right|\le C_{1}\left(\sup_{\omega}\left|{\phi}\right|^{\alpha}\right)\left(\sup_{B_{2\ell}}\left|{\phi}\right|^{1-\alpha}\right).
\end{equation*}
where $C_{1}=C(V,|\omega|)>0$ is a constant depending only on $V$ and $|\omega|$.

On the other hand, from the definition of $D_{1\ell}$, $D_{2\ell}$, $B_{1\ell}$ and $B_{2\ell}$, we observe that 
\begin{equation*}
D_{1\ell}\subset B_{1\ell}\subset B_{2\ell}\subset D_{2\ell},\quad \mbox{for all}\ \ell\in \mathbb{Z}.
\end{equation*}
Therefore, using again Lemma~\ref{le:ODE}, we conclude that 
\begin{equation*}
\begin{aligned}
\sup_{D_{1\ell}}|\phi|
&\le \left(\sup_{I_{1\ell}}\left|\varphi\right|\right) \left(\sup_{B_{1\ell}}\left|\frac{\phi}{\varphi}\right|\right)\\
&\le C_{1}e^{100\|V\|_{\infty}}\left(\sup_{\omega}\left|{\phi}\right|^{\alpha}\right)\left(\sup_{D_{2\ell}}\left|{\phi}\right|^{1-\alpha}\right).
\end{aligned}
\end{equation*} 
This completes the proof of~\eqref{est:Linfty}.

\smallskip
\textbf{Step 2.} Replacing $L^{\infty}$ norm with $L^{2}$ norm.
From~\eqref{est:Linfty}, there exists some constants $\alpha_{1}=\alpha_{1}(\Lambda,|\omega|)\in (0,1)$ and $C_{2}=C_{2}(\Lambda,|\omega|)>0$, depending only on $\Lambda$ and $|\omega|$, such that for any $\widetilde{\omega}\subset \omega$ with $  \frac{1}{2}|\omega|\le |\widetilde{\omega}|\le |\omega|$, we have 
\begin{equation}\label{est:Linfty2}
    \sup_{D_{1\ell}}|\phi|\le C_{2}\left(\sup_{\widetilde{\omega}}|\phi|^{\alpha_{1}}\right)\left(\sup_{D_{2\ell}}|\phi|^{1-\alpha_{1}}\right).
    \end{equation}

Assume $\phi \not\equiv 0$ on $D_{1\ell}$.
Let $0<\varepsilon<1$ be a small constant to be chosen later. Define 
\begin{equation*}
\delta=\varepsilon\left(\big(\sup\limits_{D_{1\ell}}|\phi|^{\frac{1}{\alpha_{1}}}\big)/\big(\sup\limits_{D_{2\ell}}|\phi|^{\frac{1}{\alpha_{1}}-1}\big)\right)\quad \mbox{and}\quad 
\omega_{\delta}=\left\{x\in \omega:|\phi(x,0)|\le \delta\right\}\subset \omega.
\end{equation*}
We claim that, there exists a constant $\varepsilon=\varepsilon(V,|\omega|)$, depending only on $V$ and $|\omega|$, such that $|\omega_{\delta}|\le \frac{1}{2}|\omega|$. Indeed, otherwise, the inequality~\eqref{est:Linfty2} would hold with $\widetilde{\omega}$ replaced by $\omega_{\delta}$ (with same constants $\alpha_{1}$ and $C_{2}$). Hence, from the definition of $\omega_{\delta}$,
\begin{equation*}
\sup_{D_{1\ell}}|\phi|\le C_{2}\left(\sup_{\omega_{\delta}}|\phi|^{\alpha_{1}}\right)
\left(\sup_{D_{2\ell}}|\phi|^{1-\alpha_{1}}\right)
\le C_{2}\delta^{\alpha_{1}}
\left(\sup_{D_{2\ell}}|\phi|^{1-\alpha_{1}}\right)
\le C_{2}\varepsilon^{\alpha_{1}}\sup_{D_{1\ell}}|\phi|.
\end{equation*}
This is a contradiction with $\phi\not\equiv 0$ on $D_{1\ell}$ for $\varepsilon$ small enough such that $C_{2}\varepsilon^{\alpha_{1}}<1$, and so we have $|\omega_{\delta}|\le \frac{1}{2}|\omega|$ for a constant $\varepsilon=\varepsilon(V,|\omega|)\in (0,1)$ which depends only on $V$ and $|\omega|$. In conclusion, using again the definition of $\omega_{\delta}$, we obtain
\begin{equation*}
\|\phi\|^{2}_{L^{2}(\omega)}\ge \|\phi\|^{2}_{L^{2}(\omega\setminus \omega_{\delta})}
\ge \frac{\delta^{2}}{2}|\omega|\ge \frac{1}{2}\varepsilon^{2}|\omega|\left(\big(\sup\limits_{D_{1\ell}}|\phi|^{\frac{2}{\alpha_{1}}}\big)/\big(\sup\limits_{D_{2\ell}}|\phi|^{\frac{2}{\alpha_{1}}-2}\big)\right),
\end{equation*}
which implies
\begin{equation*}
   \|\phi\|_{L^{2}(D_{1\ell})}\le \sup_{D_{1\ell}}|\phi|\le 2(\varepsilon^{2}|\omega|)^{-\frac{\alpha_{1}}{2}}\|\phi\|_{L^{2}(\omega)}^{\alpha_{1}}\left(\sup_{D_{2\ell}}|\phi|^{1-\alpha_{1}}\right).
\end{equation*}
The proof of Proposition~\ref{pro:L2} is complete.
\end{proof}

\subsection{Spectral inequality}\label{SS:LOW}
In this subsection, we deduce the spectral inequality and then establish the observability estimate for the 1D heat equation as a consequence.  Following Burq-Moyano~\cite{BurqMoyano}, we first establish the spectral inequality for the low-frequency part from $L^{2}$-propagation of smallness.

\begin{lemma}\label{le:spectra}
Let $\Omega$ be a $(1,\zeta)$-thick set and let $V\in C(\R)\cap L^{\infty}(\R)$ with $V\ge 1$.  Then there exists a constant $C=C(V,\zeta)>0$, depending only on $V$ and $\zeta$, such that for any $\mu>0$ and any $f\in L^{2}(\R)$, we have 
\begin{equation*}
\left\|\Pi_{\mu}f\right\|_{L^{2}(\R)}\le Ce^{3\mu}\|\Pi_{\mu}f\|_{L^{2}(\Omega)}.
\end{equation*}
\end{lemma}
\begin{proof}
Without loss of generality, we assume that $f\in L^2(\R)$ is real-valued. 
For $\mu>0$, $(x,y)\in D_{3}$ and $f\in L^{2}(\R)$, we set 
\begin{equation*}
F_{\mu}(x,y)=\int_{0}^{\mu}\cosh(y\lambda)\d m_{\lambda}f(x).
\end{equation*}
Using the fact that $(\cosh{s})''=\cosh s$ and the definition of $\d m_{\lambda}$, 
\begin{equation*}
\partial_{y}^{2}F_{\mu}={H}F_{\mu}=\int_{0}^{\mu}\lambda^{2}\cosh(y\lambda)\d m_{\lambda}f.
 \end{equation*}
 It follows that 
\begin{equation*}
-\Delta F_{\mu}+V(x)F_{\mu}=-\partial_{y}^{2}F_{\mu}+{H}F_{\mu}=0.
\end{equation*}
On the other hand, for the case of $y=0$, using the fact that $(\cosh s)'=\sinh s$,
\begin{equation*}
{\partial_{y}F_{\mu}}_{|y=0}=\int_{0}^{\mu}\lambda \sinh(0\lambda)\d m_{\lambda}f=0.
\end{equation*}
Hence, the function $F_{\mu}$ is a $H_{loc}^{2}$ solution for~\eqref{equ:ellV} with $F_{\mu}(x,0)=\Pi_{\mu}f(x)$ on $\mathbb{R}$. 

We now apply Proposition~\ref{pro:L2} to $F_{\mu}$, and thus, for any $\ell\in \mathbb{Z}$, we obtain
\begin{equation*}
\|F_{\mu}\|_{L^{2}(D_{1\ell})}\le C\|\Pi_{\mu}f\|_{L^{2}(\Omega_{\ell})}^{\alpha}\left(\sup_{D_{2\ell}}|F_{\mu}|^{1-\alpha}\right),
\end{equation*}
where $C=(V,\zeta)>0$ is a constant depending only on $V$ and $\zeta$.
Therefore, from Young's inequality for products, for any $\varepsilon$ small enough, we have
\begin{equation*}
\|F_{\mu}\|^{2}_{L^{2}(D_{1\ell})}\le \frac{C_{1}}{\varepsilon}\|\Pi_{\mu}f\|^{2}_{L^{2}(\Omega_{\ell})}+C_{1}\varepsilon\|F_{\mu}\|^{2}_{L^{\infty}{(D_{2\ell})}},
\end{equation*}
where $C_{1}=C_{1}(V,\zeta)>0$ depends only on $V$ and $\zeta$. Summing over $\ell\in \mathbb{Z}$, we find,
\begin{equation}\label{est:FmuL2}
\|F_{\mu}\|^{2}_{L^{2}(D_{1})}\le \frac{C_{1}}{\varepsilon}\|\Pi_{\mu}f\|^{2}_{L^{2}(\Omega)}+C_{1}\varepsilon\sum_{\ell\in \mathbb{Z}}\|F_{\mu}\|^{2}_{L^{\infty}{(D_{2\ell})}}.
    \end{equation}
Then we fix suitable cut-off functions. Let $\chi:\R^{2}\to \R$ be a $C^{2}$ function such that 
\begin{equation*}
    \chi\equiv 1\ \ \mbox{on}\ \left[-\frac{3}{2},\frac{3}{2}\right]^{2}\quad \mbox{and}\quad \mbox{supp}\chi\subset \left[-\frac{5}{2},\frac{5}{2}\right]^{2}.
    \end{equation*}
    For any $\ell\in \mathbb{Z}$, we set 
    \begin{equation*}
        \chi_{\ell}(x,y)=\chi\left(x-\ell-\frac{1}{2},y\right)
        \Longrightarrow \chi_{\ell}\equiv 1\ \ \mbox{on}\ D_{2\ell}\quad \mbox{and}\quad \mbox{supp}\chi_{\ell}\subset D_{3\ell}.
    \end{equation*}
    Therefore, using the 2D Sobolev embedding theorem, we deduce that 
    \begin{equation}\label{est:FmuLinfty}
        \begin{aligned}
\sum_{\ell\in \mathbb{Z}}\|F_{\mu}\|^{2}_{L^{\infty}(D_{2\ell})}
&\le \pi\sum_{\ell\in \mathbb{Z}}\|\chi_{\ell} F_{\mu}\|^{2}_{H^{2}(\R^{2})}\\
&\le {{C}_{2}}\|F_{\mu}\|^{2}_{H^{2}(D_{3})}\le C_{3}(1+\mu^{4})\|F_{\mu}\|^{2}_{L^{2}(D_{3})},
\end{aligned}    
\end{equation}
where ${C}_{2}>0$ and $C_{3}>0$ depend only on $V$ and the chose of $\chi$.
Combining~\eqref{est:FmuL2} and~\eqref{est:FmuLinfty}, for any $\varepsilon$ small enough, we have 
\begin{equation*}
\|F_{\mu}\|^{2}_{L^{2}(D_{1})}\le \frac{C_{1}}{\varepsilon}\|\Pi_{\mu}f\|_{L^{2}(\Omega)}^{2}+C_{1}C_{3}\varepsilon(1+\mu^{4})\|F_{\mu}\|_{L^{2}(D_{3})}^{2}.
\end{equation*}

On the other hand, from the definition of the spectral projector $\Pi_{\mu}$,
\begin{equation*}
\begin{aligned}
\|F_{\mu}\|_{L^{2}(D_{1})}^{2}
&=\int_{-\frac{1}{2}}^{\frac{1}{2}}\int_{0}^{\mu}\cosh^{2}(y\lambda)(\d m_{\lambda}f,f)_{L^{2}(\R)}\d y\ge \|\Pi_{\mu}f\|_{L^{2}(\R)}^{2},\\
\|F_{\mu}\|_{L^{2}(D_{3})}^{2}
&=\int_{-\frac{5}{2}}^{\frac{5}{2}}\int_{0}^{\mu}\cosh^{2}(y\lambda)(\d m_{\lambda}f,f)_{L^{2}(\R)}\d y\le 5e^{5\mu}\|\Pi_{\mu}f\|_{L^{2}(\R)}^{2}.
\end{aligned}
\end{equation*}
Gathering the above three inequalities, we obtain
\begin{equation*}
\|\Pi_{\mu}f\|_{L^{2}(\R)}^{2}\le \frac{C_{1}}{\varepsilon}\|\Pi_{\mu}f\|_{L^{2}(\Omega)}^{2}
+5C_{1}C_{3}\varepsilon e^{5\mu}(1+\mu^{4})\|\Pi_{\mu}f\|_{L^{2}(\R)}^{2},
\end{equation*}
which completes the proof of Lemma~\ref{le:spectra} by taking $\varepsilon$ small enough.
\end{proof}

\begin{remark}
	The assumption $V\geq 1$ can be removed in Lemma \ref{le:spectra}, thus extending the result of Lebeau-Moyano \cite{LeM} to bounded potentials. Here, we consider the spectral measure for $H$ instead of $\sqrt{H}$. The additional argument can be sketched as follows. Although $V(x)$ may take negative values, by Sturm-Liouville theory, we are still able to construct $\varphi(x)$, bounded from below and above as in Lemma \ref{le:ODE} by restricting the size of interval $I$ such that $|I|\leq 2\sigma_0$, where $\sigma_0=1/4\pi\sqrt{\|V\|_{L^{\infty}}}$.  Consequently, the analogue of Proposition \ref{pro:L2} remains true if we replace $D_{1\ell}, D_{2\ell}$ with boxes of the form  $I_1\times\left[-\frac{1}{2},\frac{1}{2}\right], I_2\times\left[-\frac{3}{2},\frac{3}{2}\right]$ such that $I_1=[x_0,x_0+\sigma_0], I_2=[x_0-\sigma_0, x_0+2\sigma_0]$.  By dividing the interval $[\ell,\ell+1]$ into at most $\left\lfloor1/\sigma_0\right\rfloor+1$ intervals of size $\sigma_0$ and using the pigeonhole principle, finally we are able to obtain the same statement as Proposition \ref{pro:L2}.  	
\end{remark}

Note that, from Lemma~\ref{le:spectra} and a standard argument, we directly have the observability estimate for 1D homogeneous heat equation (not necessarily real-valued)
\begin{equation}\label{equ:1D heat}
    \partial_{t}u-\partial_{x}^{2}u+V(x)u=0,\quad  u_{|t=0}=u_{0}\in L^{2}(\R).
\end{equation}
Recall that, the operator $-H$ generates a semigroup $e^{-tH}$ and so the solution for~\eqref{equ:1D heat} can be written as $u(t)=e^{-t{H}}u_{0}\in L^{2}(\R)$.

\begin{proposition}\label{prop:heat}
    Let $\Omega$ be a $(1,\zeta)$-thick set and let $V\in C(\R)\cap L^{\infty}(\R)$ with $V\ge 1$. 
Then there exists a constant $C=C(V,\zeta)>0$, depending only on $V$ and $\zeta$, such that for any $T>0$ and any solution $u$ of~\eqref{equ:1D heat} we have 
  \begin{equation*}
  \|u(T)\|_{L^{2}(\R)}^{2}\le Ce^{\frac{C}{T}}\int_{0}^{T}\left\|u(t)\right\|^{2}_{L^{2}(\Omega)}\d t.
  \end{equation*}
  \end{proposition}

  \begin{proof}
  For the sake of completeness and the readers' convenience, the details of the proof for Proposition~\ref{prop:heat} is given in Appendix~\ref{AppB}.  
  \end{proof}

Thanks to the above Proposition, we obtain the observability estimate for the 1D inhomogeneous heat equation
\begin{equation}\label{equ:1Dinhomoheat}
    \partial_{t}u-\partial_{x}^{2}u+V(x)u=F\in L^{2}((0,\infty):L^{2}(\R)),\quad  u_{|t=0}=u_{0}\in H^{2}(\R).
    \end{equation}

\begin{corollary}\label{coro:1Dheat}
   Let $\Omega$ be a $(1,\zeta)$-thick set and let $V\in C(\R)\cap L^{\infty}(\R)$ with $V\ge 1$. 
Then there exists a constant $C=C(V,\zeta)>0$, depending only on $V$ and $\zeta$, such that for any $T>0$ and any solution $u$ of~\eqref{equ:1Dinhomoheat} we have 
  \begin{equation*}
  \|u(T)\|_{L^{2}(\R)}^{2}\le Ce^{\frac{C}{T}}\int_{0}^{T}\left(\left\|Hu(t)\right\|^{2}_{L^{2}(\Omega)}+\|F(t)\|^{2}_{L^{2}(\R)}\right)\d t.
  \end{equation*}
  \end{corollary}

\begin{proof}
    We decompose the solution $u$ as 
    \begin{equation*}
        u(t,x)=u_{1}(t,x)+u_{2}(t,x),\quad \mbox{on}\ [0,T]\times \R,
    \end{equation*}
    where $u_{1}$ and $u_{2}$ are the solutions for the following 1D homogeneous or inhomogeneous heat equations
    \begin{equation}\label{equ:u1u2}
        \left\{
\begin{aligned}
    \partial_{t}u_{1}-\partial_{x}^{2}u_{1}+V(x)u_{1}&=0,\quad {u_{1}}_{|t=0}=u_{0},\\
    \partial_{t}u_{2}-\partial_{x}^{2}u_{2}+V(x)u_{2}&=F,\quad {u_{2}}_{|t=0}=0.\end{aligned}
        \right.
    \end{equation}
    First, from $V\in C(\R)\cap L^{\infty}(\R)$ with $V\ge 1$, we have 
    \begin{equation*}
        H\ge {\rm{Id}}\Longrightarrow \left\|Hu_{1}\right\|_{L^{2}(\R)}\ge \|u_{1}\|_{L^{2}(\R)}.
    \end{equation*}
    It follows from Proposition~\ref{prop:heat} that 
    \begin{equation*}
        \|u_{1}(T)\|_{L^{2}(\R)}^{2}\le \|Hu_{1}(T)\|_{L^{2}(\R)}^{2}\le Ce^{\frac{C}{T}}\int_{0}^{T}\|Hu_{1}(t)\|^{2}_{L^{2}(\Omega)}\d t,
        \end{equation*}
        where $C=C(V,\zeta)>0$ depends only on $V$ and $\zeta$. Note that, from~\eqref{equ:u1u2}, the term $Hu_{1}$ can be rewritten as $Hu_{1}=Hu+\partial_{t}u_{2}-F$ and so we obtain
        \begin{equation}\label{est:u1}
        \begin{aligned}
             \|u_{1}(T)\|_{L^{2}(\R)}^{2}
             &\le 3Ce^{\frac{C}{T}}\int_{0}^{T}\|Hu(t)\|^{2}_{L^{2}(\Omega)}\d t\\
            &+3Ce^{\frac{C}{T}}\int_{0}^{T}\left(\|\partial_{t}u_{2}(t)\|^{2}_{L^{2}(\R)}+\|F(t)\|^{2}_{L^{2}(\R)}      
            \right)\d t.  
            \end{aligned}
             \end{equation}
             Second, using again~\eqref{equ:u1u2}, we directly have 
             \begin{equation*}
                (\pt u_{2})^{2}+\frac{1}{2}\partial_{t}\left((\partial_{x}u_{2})^{2}+V(x)u_{2}^{2}\right)-\partial_{x}\left((\partial_{t}u_{2})(\partial_{x}u_{2})\right)=F\partial_{t}u_{2}.
             \end{equation*}
             Integrating the above identities over $[0,T]\times \R$, and then using Cauchy-Schwarz inequality, we see that 
             \begin{equation}\label{est:u2}
                    \|u_{2}(T)\|_{L^{2}(\R)}^{2}+\int_{0}^{T}\|\partial_{t}u_{2}(t)\|_{L^{2}(\R)}^{2}\d t\le \int_{0}^{T}\|F(t)\|_{L^{2}(\R)}^{2}\d t.
             \end{equation}
             Here, we used the fact that $H\ge {\rm{Id}}$ and the zero initial condition of $u_2$ in $H^2(\R)$. 
             Combining~\eqref{est:u1} and~\eqref{est:u2} with Cauchy-Schwarz inequality, we obtain
             \begin{equation*}
             \begin{aligned}
                 \|u(T)\|_{L^{2}(\R)}^{2}
                 &\le 2\left(\|u_{1}(T)\|_{L^{2}(\R)}^{2}+\|u_{2}(T)\|_{L^{2}(\R)}^{2}\right)\\
                 &\le (6C+1)e^{\frac{C}{T}}\int_{0}^{T}\left(\|Hu(t)\|^{2}_{L^{2}(\Omega)}+2\|F(t)\|^{2}_{L^{2}(\R)}\right)\d t.         
                 \end{aligned}
                 \end{equation*}
        The proof of Proposition~\ref{prop:heat} is complete.
\end{proof}

For the notational convenience of introducing the FBI transformation later, we can reverse the time $t$ to $T-t$ to obtain the following observability estimate for the 1D inhomogeneous backward heat equation
\begin{equation}\label{equ:1Dinhoback}
    \partial_{t}u+\partial_{x}^{2}u-V(x)u=F\in L^{2}((0,\infty):L^{2}(\R)),\quad  u_{|t=T}=u_{T}\in H^{2}(\R).
    \end{equation}

\begin{corollary}\label{coro:1Dheatback}
   Let $\Omega$ be a $(1,\zeta)$-thick set and let $V\in C(\R)\cap L^{\infty}(\R)$ with $V\ge 1$. 
Then there exists a constant $C=C(V,\zeta)>0$, depending only on $V$ and $\zeta$, such that for any $T>0$ and any solution $u$ of~\eqref{equ:1Dinhoback} we have 
  \begin{equation*}
  \|u(0)\|_{L^{2}(\R)}^{2}\le Ce^{\frac{C}{T}}\int_{0}^{T}\left(\left\|Hu(t)\right\|^{2}_{L^{2}(\Omega)}+\|F(t)\|^{2}_{L^{2}(\R)}\right)\d t.
  \end{equation*}
  \end{corollary}
    
%Note that, from Lemma~\ref{le:spectra} and conservation of $L^{2}$ along the Schr\"odinger flow, we directly have the following observability inequality for the low-frequency part.

%\begin{corollary}\label{coro:Lowobser}
%Let $\Omega$ be a thick set and let $V\in C(\R)\cap L^{\infty}(\R)$ with $V\ge 0$. Then for any $T>0$ and $\mu>0$, there exists a constant $C=(T,V,\Omega,\mu)$, depending only on 
%$T,V,\Omega$ and $\mu$,
%such that for any $f\in L^{2}(\R)$, we have 

%\begin{equation*}
%\|\Pi_{\mu}f\|_{L^{2}(\R)}^{2}\le C\int_{0}^{T}\|e^{it{H}}\Pi_{\mu}f\|_{L^{2}(\Omega)}^{2}\d t.
%\end{equation*}
%\end{corollary}

\section{Resolvent estimate}\label{SS:RESO}
In this section, we recall the resolvent estimate for 1D Schr\"odinger operator $H$ and then deduce the observability estimate for the high-frequency part. Recall that, the resolvent estimate for operator $H=-\partial_x^2$ was first given in~\cite[Proposition 1]{Gr}. 

\begin{lemma}[\cite{Gr}]\label{le:resolvent}
Let $\Omega$ be a $(1,\zeta)$-thick set and let $V\in L^{\infty}(\R)$. Then there exist some constants $\mu_{0}=\mu_{0}(V,\zeta)$ and $C=C(V,\zeta)$, depending only on $V$ and $\zeta$, such that for any $\mu>\mu_{0}$ and any $f\in H^{2}(\R)$, we have 
\begin{equation*}
\|f\|_{L^{2}(\R)}^{2}\le \frac{C}{\mu}\|({H}-\mu)f\|_{L^{2}(\R)}^{2}+C\|f\|_{L^{2}(\Omega)}^{2}.
\end{equation*}
\end{lemma}

For the reader's convenience, we will provide a proof of Lemma~\ref{le:resolvent}, different from the one given in~\cite{Gr}. We first need the following technical estimate.
\begin{lemma}\label{le:technical}
    Let $0<\zeta<1$. Then there exists a constant $c=c(\zeta)>0$, depending only on $\zeta$, such that for any measurable set $\omega\subset (0,1)$ with $\zeta\le |\omega|\le 1$ and any $\lambda>1$, we have 
    \begin{equation}\label{est:A}
       \inf_{x_{0}\in \R} \int_{\omega}\cos^{2}(\lambda(x-x_{0}))\d x\ge c(\zeta).
    \end{equation}
\end{lemma}

\begin{remark}
    We mention here that Lemma~\ref{le:technical} is not a direct consequence of the Riemann-Lebesgue lemma and the proof will require more quantitative analysis, since the lower bound in estimate~\eqref{est:A} only depends on the size of the measurable set $\omega$ which is different from the statement of Riemann-Lebesgue lemma.
    \end{remark}

\begin{proof}[Proof of Lemma~\ref{le:technical}]
Using the structure of a measurable set in $\R$, for any $0<\varepsilon<1$, there exists a finite union of disjoint open intervals 
\begin{equation*}
    U=\bigcup\limits_{n=1}^{N}I_{n}\quad \mbox{with}\quad 
    I_{n}=(a_{n},b_{n})\subset (0,1) \  \ \mbox{for any}\ n\in \left\{1,\dots,N\right\},
    \end{equation*}
    such that $|\omega \setminus U|+|U\setminus \omega |<\varepsilon$ and $|U|\ge \zeta-\varepsilon$.
    It follows that 
    \begin{equation}\label{est:cos2}
    \begin{aligned}
       \int_{\omega}\cos^{2}(\lambda(x-x_{0}))\d x
       &\ge 
       \sum_{n=1}^{N}\int_{a_{n}}^{b_{n}}\cos^{2}(\lambda(x-x_{0}))\d x-\varepsilon.
       \end{aligned}
        \end{equation}
By an elementary computation, on any finite interval $I_{n}=(a_{n},b_{n})$, we have 
\begin{equation}\label{est:ajbj}
\begin{aligned}
    &\int_{a_{n}}^{b_{n}}\cos^{2}(\lambda(x-x_{0}))\d x\\
    &=\frac{1}{2}\int_{a_{n}}^{b_{n}}(1+\cos (2\lambda(x-x_{0})))\d x\\
    &=\frac{1}{2}\left((b_{n}-a_{n})+\frac{1}{\lambda}\sin \left(\lambda(b_{n}-a_{n})\right)\cos \left(\lambda(b_{n}+a_{n}-2x_{0})\right)\right).
    \end{aligned}
    \end{equation}
Let $0<\delta\ll 1$ be a small constant to be chosen later. For any $m\in \mathbb{Z}$, we set 
\begin{equation*}
    J_{m,\delta}=\left(x_{0}+\frac{m \pi }{2\lambda}-\frac{\delta}{2\lambda},x_{0}+\frac{m \pi }{2\lambda}+\frac{\delta}{2\lambda}\right)\cap (0,1)\ \ \mbox{and}\  \ S=\bigcup\limits_{m\in \mathbb{Z}}J_{m,\delta}.
    \end{equation*}
Note that, $\left\{J_{m,\delta}\right\}_{m\in \mathbb{Z}}$ are disjoint sets.
For further reference, we consider
\begin{equation*}
    \mathcal{A}=\left\{m\in \mathbb{Z}:J_{m,\delta}\ne \emptyset\right\},\quad \mbox{and thus we have}\quad  \# \mathcal{A}\le 3\lambda.
\end{equation*}
We split the index set $\left\{1,\dots,N\right\}$ to the following three cases according to the size and position of the finite interval $I_{n}$ and then establish estimate for each case.

\smallskip
\textbf{Case 1.} Let $(b_{n}-a_{n})\ge \frac{\delta}{\lambda}$. From the fact that $\frac{\sin x}{x}$ is decreasing on $[\delta,\pi)$, for $0<\delta\ll 1$, 
\begin{equation*}
\sup_{[\delta,\infty)}\left|\frac{\sin x}{x}\right|\le \frac{\sin \delta}{\delta} \Longrightarrow \frac{1}{\lambda}\left|\sin \left(\lambda(b_{n}-a_{n})\right)\right|\le \frac{\sin \delta}{\delta}(b_{n}-a_{n}).
\end{equation*}
Based on the above estimate and~\eqref{est:ajbj}, we obtain
\begin{equation}\label{est:case1}
\int_{a_{n}}^{b_{n}}\cos^{2}(\lambda(x-x_{0}))\d x
\ge \frac{1}{2}\left(1-\frac{\sin \delta}{\delta}\right)(b_{n}-a_{n}).
\end{equation}

\textbf{Case 2.} Let $0<(b_{n}-a_{n})< \frac{\delta}{\lambda}$ with $S\cap I_{n}= \emptyset$. First, from the definition of $S$ and $S\cap I_{n}=\emptyset$, there exists $m\in \mathbb{Z}$ such that 
\begin{equation*}
    I_{n}\subset \left(x_{0}+\frac{m \pi }{2\lambda}+\frac{\delta}{2\lambda},x_{0}+\frac{ (m+1)\pi}{2\lambda}-\frac{\delta}{2\lambda}\right),
    \end{equation*}
    and thus, from $\sin^{2}x+\cos^{2}x=1$, we find
    \begin{equation*}
    \begin{aligned}
   \mbox{dist}(\lambda(b_{n}+a_{n}-2x_{0}),\pi\mathbb{Z})\ge \delta\Longrightarrow
    \left|\cos (\lambda(b_{n}+a_{n}-2x_{0})\right|\le \sqrt{1-\sin^{2}\delta}.
    \end{aligned}
    \end{equation*}
    Therefore, using again~\eqref{est:ajbj} and $|\sin x|\le |x|$, we obtain
    \begin{equation}\label{est:case2}
        \int_{a_{n}}^{b_{n}}\cos^{2}(\lambda(x-x_{0}))\d x\ge
        \frac{1}{2}\left(1-\sqrt{1-\sin^{2}\delta}\right)(b_{n}-a_{n}).
        \end{equation}

\textbf{Case 3.} We now consider the last case, that is, the case of $0<(b_{n}-a_{n})< \frac{\delta}{\lambda}$ with $S\cap I_{n}\ne  \emptyset$. To simplify notation, we denote 
\begin{equation*}
    \mathcal{B}=\left\{n\in\left\{1,\dots,N\right\}:0<(b_{n}-a_{n})< \frac{\delta}{\lambda}\ \ \mbox{with} \ \ S\cap I_{n}\ne  \emptyset\right\}.
\end{equation*}

Note that, for any $n\in \mathcal{B}$, there exists $m\in\mathcal{A}$ such that 
\begin{equation*}
I_{n}\subset {J}_{m,3\delta}\quad \mbox{where}\ \ 
{J}_{m,3\delta}=\left(x_{0}+\frac{m \pi }{2\lambda}-\frac{3\delta}{2\lambda},x_{0}+\frac{m \pi }{2\lambda}+\frac{3\delta}{2\lambda}\right)\cap(0,1).
\end{equation*}
Next, from $0<\delta\ll 1$, for any $(m,m')\in \mathbb{Z}^{2}$ with $m\ne m'$, we find ${J}_{m,3\delta}\cap
{J}_{m',3\delta}=\emptyset$.
Therefore, using the fact that $\left\{I_{n}\right\}_{n=1}^{N}$ are disjoint intervals and $\# \mathcal{A}\le 3\lambda$, we obtain
    \begin{equation}\label{est:case3}
   \sum_{n\in \mathcal{B}}(b_{n}-a_{n})=\bigg|\bigcup\limits_{n\in \mathcal{B}}I_{n}\bigg|\le \bigg|\bigcup\limits_{m\in \mathcal{A}}{J}_{m,3\delta}\bigg|\le \frac{3\delta}{\lambda}\#\mathcal{A}\le  9\delta.
   \end{equation}

   Combining~\eqref{est:cos2},~\eqref{est:case1},~\eqref{est:case2},~\eqref{est:case3} with $|U|\ge \zeta-\varepsilon$, we conclude that 
   \begin{equation*}
    \inf_{x_{0}\in \R}\int_{\omega}\cos^{2}(\lambda(x-x_{0}))\d x\ge \frac{1}{2}\left(1-\max\left(\frac{\sin \delta}{\delta},\sqrt{1-\sin^{2}\delta}\right)\right)(\zeta-\varepsilon-9\delta)-\varepsilon.
   \end{equation*}
   We see that~\eqref{est:A} follows from the above estimate fo $\varepsilon$ and $\delta$ small enough.
\end{proof}

We now give an alternative proof of Lemma~\ref{le:resolvent} for the reader's convenience.

\begin{proof}[Proof of Lemma~\ref{le:resolvent}]

\textbf{Step 1.} Estimate for the flat case. Let $V=0$ and $\ell\in \mathbb{Z}$. For $\mu>1$, we denote
\begin{equation*}
    F=-\partial_{x}^{2}f-\mu f\quad \mbox{on}\ \R.
\end{equation*}

For any $s\in I_{1\ell}$, the function $f$ can be expressed by 
\begin{equation*}
    f(x)=\cos (\sqrt{\mu}(x-s))f(s)+\frac{\sin  (\sqrt{\mu}(x-s))}{\sqrt{\mu}}f'(s)-\int_{s}^{x}\frac{\sin  (\sqrt{\mu}(x-y))}{\sqrt{\mu}}F(y)\d y.
    \end{equation*}

Note that, in the above identity, the sum of the first two terms can be rewritten as
\begin{equation*}
    \cos (\sqrt{\mu}(x-s))f(s)+\frac{\sin  (\sqrt{\mu}(x-s))}{\sqrt{\mu}}f'(s)=r\cos \left(\sqrt{\mu}\left(x-s-\frac{\theta}{\sqrt{\mu}}\right)\right),
    \end{equation*}
    where
    \begin{equation*} 
        \theta\in [0,2\pi)\quad \mbox{and}\quad 
           r=\sqrt{|f(s)|^{2}+\mu^{-1}|f'(s)|^{2}}.
\end{equation*}
Therefore, from Lemma~\ref{le:technical} and $\zeta\le |\Omega_{\ell}|\le 1$, there exists a constant $c=c(\zeta)>0$, depending only on $\zeta$, such that
    \begin{equation*}
        c|f(s)|^{2}\le \left\| \cos (\sqrt{\mu}(x-s))f(s)+\frac{\sin  (\sqrt{\mu}(x-s))}{\sqrt{\mu}}f'(s)\right\|_{L^{2}(\Omega_{\ell})}^{2}.
        \end{equation*}
Combining the above estimate with the expansion of $f(x)$, we find
\begin{equation*}
\begin{aligned}
    c|f(s)|^{2}&\le \left\|f(x)+\int_{s}^{x}\frac{\sin  (\sqrt{\mu}(x-y))}{\sqrt{\mu}}F(y)\d y\right\|^{2}_{L^{2}(\Omega_{\ell})}\\
    &\le 2\left\|f\right\|_{L^{2}(\Omega_{\ell})}^{2}+\frac{2}{\mu}\left\|\int_{s}^{x}|F(y)|\d y\right\|^{2}_{L^{2}(\Omega_{\ell})}\\
    &\le 2\|f\|_{L^{2}(\Omega_{\ell})}^{2}+\frac{2}{\mu}\|F\|^{2}_{L^{2}(I_{1\ell})}.
   \end{aligned}
\end{equation*}
Integrating the above estimate with respect to variable $s$ over $I_{1\ell}$ and then summing over $\ell\in \mathbb{Z}$, we conclude that 
\begin{equation}\label{est:V=0}
    \|f\|_{L^{2}(\R)}^{2}\le \frac{C}{\mu}\|(-\partial_{x}^{2}-\mu)f\|_{L^{2}(\R)}^{2}+C\|f\|_{L^{2}(\Omega)}^{2},
    \end{equation}
    where $C=C(\zeta)>0$ is a constant depending only on $\zeta$.

\smallskip
    \textbf{Step 2.} Conclusion. Note that, from~\eqref{est:V=0} and $V\in L^{\infty}(\R)$, we obtain
    \begin{equation*}
    \begin{aligned}
        \|f\|_{L^{2}(\R)}^{2}
        &\le  \frac{C}{\mu}\|(-\partial_{x}^{2}-\mu)f\|_{L^{2}(\R)}^{2}+C\|f\|_{L^{2}(\Omega)}^{2}\\
        &\le \frac{C}{\mu}\|(H-\mu)f\|_{L^{2}(\R)}^{2}
        +\frac{C}{\mu}\|V\|_{\infty}^{2}\|f\|^{2}_{L^{2}(\R)}
        +C\|f\|_{L^{2}(\Omega)}^{2},        
        \end{aligned}
        \end{equation*}
        which completes the proof of Lemma~\ref{le:resolvent} by taking $\mu$ large enough.
\end{proof}

Combining the above resolvent estimate with an argument in~\cite[Section 3]{BZ1}, we obtain the following observability inequality for the high-frequency part.

\begin{corollary}\label{coro:High}
Let $\Omega$ be a $(1,\zeta)$-thick set and let $V\in L^{\infty}(\R)$ with $V\ge 1$. 
Then there exist some constants $\mu_{1}=\mu_{1}(V,\zeta)$ and  $C=C(V,\zeta)$, depending only on $V$ and $\zeta$, such that for any $T>0$, $\mu>\mu_{1}\left(1+T^{-1}\right)$ and $f\in L^{2}(\R)$, we have 
\begin{equation*}
\|(1-\Pi_{\mu})f\|_{L^{2}(\R)}^{2}\le \frac{C}{T}\int_{0}^{T}\left\|e^{it{H}}(1-\Pi_{\mu})f\right\|_{L^{2}(\Omega)}^{2}\d t.
\end{equation*}

\end{corollary}

\begin{proof}
Let $f\in H^{2}(\R)$ and let $F$ be the solution of the 1D Schr\"odinger equation
\begin{equation*}
i\partial_{t}F-\partial_{x}^{2}F+V(x)F=0\quad \mbox{with}\ F_{|t=0}=(1-\Pi_{\mu})f\in H^{2}(\R).
\end{equation*}
We fix a cut-off $C^{2}$ function $\chi:\R\to [0,1]$ satisfying
\begin{equation*}
    \chi\equiv 1 \ \mbox{on}\ \left[\frac{1}{4},\frac{3}{4}\right],\quad \mbox{supp}\chi\subset[0,1]
    \quad \mbox{and}\quad \chi'\in [-5,5].
\end{equation*}
For time $T>0$, we consider a new function
\begin{equation*}
    \Psi(t,x)=\chi\left(\frac{t}{T}\right)F(t,x)
    \Longrightarrow i\partial_{t}\Psi-\partial_{x}^{2}\Psi+V(x)\Psi=\frac{i}{T}\chi'\left(\frac{t}{T}\right)F.
    \end{equation*}
Taking the Fourier transform in the above equation with respect to $t$,
\begin{equation}\label{equ:H2pixi}
(H- \xi)\widehat{\Psi}(\xi,x)=\frac{i}{T}\mathcal{F}_{t\to \xi}\left({\chi'\left(\frac{t}{T}\right)F}\right)(\xi,x).
\end{equation}
Let $\mu^2>2(\mu_{0}+2)$ where $\mu_{0}$ is the parameter appearing in Lemma~\ref{le:resolvent}. On the one hand, we apply Lemma~\ref{le:resolvent} to $\widehat{\Psi}$. Hence, based on the identity~\eqref{equ:H2pixi}, for $\xi>\frac{\mu^2}{2}>\mu_{0}+2$, we directly have
\begin{equation}\label{est:Psi1}
    \left\|\widehat{\Psi}(\xi,x)\right\|^{2}_{L_{x}^{2}(\R)}\le \frac{C}{\mu^2T^2}\left\|\mathcal{F}_{t\to \xi}\left({\chi'\left(\frac{t}{T}\right)F}\right)(\xi,x)\right\|_{L_{x}^{2}(\R)}^{2}+C\left\|\widehat{\Psi}(\xi,x)\right\|^{2}_{L_{x}^{2}(\Omega)},
    \end{equation}
    where $C=C(V,\zeta)$ is a constant depending only on $V$ and $\zeta$.

\smallskip
    On the other hand, for $\xi\le \frac{\mu^2}{2}$, we estimate \eqref{equ:H2pixi} directly using the fact that $(H-\xi)(\mathrm{Id}-\Pi_{\mu})$ is invertible.  More precisely, from the definition of $\d m_{\lambda}$ and $\Psi$, we have
    \begin{equation*}
       (H- \xi)\widehat{\Psi}(\xi,x)=\int_{\mu}^{\infty}(\lambda^{2}- \xi)\d m_{\lambda}\widehat{\Psi}(\xi,x).        \end{equation*}
 Observe that $\frac{\mu^2}{2}<\lambda^2-\xi$ for $\lambda\geq \mu$ and $\xi\leq \frac{\mu^2}{2}$.  Therefore, for any $\xi\le \frac{\mu^2}{2}$,  we see that
    \begin{equation*}
    \begin{aligned}
    \frac{1}{2}\mu^{2}\left\|\widehat{\Psi}(\xi,x)\right\|^{2}_{L_{x}^{2}(\R)}
         &\le \int_{\mu}^{\infty}(\lambda^{2}-\xi)\left(\d m_{\lambda}\widehat{\Psi}(\xi,x),\widehat{\Psi}(\xi,x)\right)_{L_{x}^{2}(\R)} \\
         &\le \left((H- \xi)\widehat{\Psi}(\xi,x),\widehat{\Psi}(\xi,x)\right)_{L_{x}^{2}(\R)}.      \end{aligned}
         \end{equation*}
    Combining the above inequality with~\eqref{equ:H2pixi}, for $\xi \le \mu$, we obtain
    \begin{equation}\label{est:Psi2}
        \left\|\widehat{\Psi}(\xi,x)\right\|^{2}_{L_{x}^{2}(\R)}\le \frac{4}{\mu^{4} T^{2}}\left\|\mathcal{F}_{t\to \xi}\left({\chi'\left(\frac{t}{T}\right)F}\right)(\xi,x)\right\|_{L_{x}^{2}(\R)}^{2}.
        \end{equation}

        \smallskip
        Gathering~\eqref{est:Psi1} and~\eqref{est:Psi2}, and then integrating over $\R$ for the variable $\xi$, we find
        \begin{equation*}
        \begin{aligned}
            \left\|\widehat{\Psi}(\xi,x)\right\|^{2}_{L_{}^{2}(\R^{2})}
            &\le \frac{C}{\mu^2T^2}\left\|\mathcal{F}_{t\to \xi}\left({\chi'\left(\frac{t}{T}\right)F}\right)(\xi,x)\right\|_{L^{2}(\R^{2})}^{2}\\
            &+\frac{4}{\mu^{4}T^{2}}\left\|\mathcal{F}_{t\to \xi}\left({\chi'\left(\frac{t}{T}\right)F}\right)(\xi,x)\right\|_{L^{2}(\R^{2})}^{2}            
            +C\left\|\widehat{\Psi}(\xi,x)\right\|^{2}_{L^{2}(\R \times\Omega)}.
            \end{aligned}
            \end{equation*}
Then, using the Plancherel theorem for the variables $t$ and $\xi$, 
\begin{equation*}
    \left\|\Psi(t,x)\right\|_{L^{2}(\R^{2})}^{2}\le 
    \left(\frac{C}{\mu^2T^2}+\frac{4}{\mu^{4}T^{2}}\right)\left\|\chi'\left(\frac{t}{T}\right)F(t,x)\right\|_{L^{2}(\R^{2})}^{2}+  
    C\left\|\Psi(t,x)\right\|_{L^{2}(\R\times \Omega)}^{2}.  
    \end{equation*}
    Therefore, by the conservation of $L_{x}^{2}$ for $F$ and the definition of $\chi$ and $\Psi$, we obtain
    \begin{equation*}
    \begin{aligned}
        \|(1-\Pi_{\mu})f\|_{L^{2}(\R)}^{2}
        &\le 50\left(\frac{C}{\mu^2T^2}+\frac{4}{\mu^{4}T^{2}}\right)\|(1-\Pi_{\mu})f\|_{L^{2}(\R)}^{2}\\
        &+\frac{2C}{T}\int_{0}^{T}\|e^{itH}(1-\Pi_{\mu})f\|_{L^{2}(\Omega)}^{2}\d t,
        \end{aligned}
        \end{equation*}
        which completes the proof if $f\in H^{2}(\R)$ by taking $\mu>\mu_1(1+T^{-1})$ large enough. Last, using a density argument, we complete the proof of Corollary~\ref{coro:High} for any $f\in L^{2}(\R)$.
\end{proof}

\section{Proof of Theorem~\ref{thm:main1}}\label{SS:MAIN}

In this section, we prove Theorem~\ref{thm:main1}. The proof is based on the general strategy introduced in Phung~\cite{Phung01} (inspired by Lebeau-Robbiano~\cite{LeRo}) for the Schr\"odinger equation in a similar context. We start with the following quantitative unique continuation estimate for the 1D Schr\"odinger equation which plays a crucial role in our proof for Theorem~\ref{thm:main1}. The key idea of the proof is to take an FBI transformation that transfers the 1D Sch\"odinger equation to the 1D heat equation.
\begin{proposition}\label{prop:key}
    Let $\Omega$ be a $(1,\zeta)$-thick set and let $V\in C(\R)\cap L^{\infty}(\R)$ with $V\ge 1$. Then there exist some constants $h_{0}=h_{0}(V,\zeta)\in (0,1)$ and $C=C(V,\zeta)>0$, depending only on $V$ and $\zeta$, such that for any $T>0$, $0<h<h_{0}\left(1+T^{-3}\right)^{-1}$ and $f\in H^{2}(\R)$, we have 
    \begin{equation}\label{est:Key}
    \begin{aligned}
       \|f\|_{L^{2}(\R)}^{2}
       &\le Ch\|Hf\|_{L^{2}(\R)}^{2}\\      
       &+Ce^{\frac{2T^{2}}{h}+\frac{C}{T}}\int_{0}^{T}\left\|He^{itH}f\right\|_{L^{2}(\Omega)}^2\d t.
       \end{aligned}
    \end{equation}
\end{proposition}

\begin{proof}
\textbf{Step 1.} FBI transformation. Following Zworski~\cite{Zworski}, we introduce the definition of FBI transformation. For $0<h<1$, $z=\tau+is\in \mathbb{C}$ and $L^{2}$-valued regular function $\Gamma(t)$, we define
\begin{equation*}
    \mathcal{T}_{h}\Gamma(z)=\frac{2^{\frac{1}{4}}}{(2\pi h)^{\frac{3}{4}}}\int_{\R}e^{-\frac{(z+t)^{2}}{2h}}\Gamma(t)\d t.
\end{equation*}
By an elementary computation and integration by parts, we directly have 
\begin{equation}\label{equ:heatSchro}
    (\partial_{s}+\partial_{x}^{2}-V(x))\mathcal{T}_{h}\Gamma(z)=
    -\mathcal{T}_{h}\left(\left(i\pt-\partial_{x}^{2}+V(x)\right)\Gamma\right)(z).
\end{equation}
This is the key point to transfer the observability estimate for 1D Schr\"odinger equation to the observability estimate for the 1D heat equation.

Fix $T>0$. We define a cut-off $C^{1}$ function $\chi:\R\to [0,1]$ satisfying
\begin{equation}\label{def:chi}
    \chi\equiv 1\ \mbox{on}\ [2T,8T],\quad \mbox{supp}\chi\subset [0,10T]\quad \mbox{and} \quad 
    \chi' \in \left[-\frac{2}{T},\frac{2}{T}\right].
\end{equation}
To simplify notation, we will prove~\eqref{est:Key} is true for $10T$ and then based on the arbitrary choice of $T$, we can complete the proof of~\eqref{est:Key} for any $T>0$. For any $f\in H^{2}(\R)$, we denote $F=e^{itH}f$ and $\widetilde{F}=\chi F$. We directly have 
\begin{equation*}
    i\partial_{t}\widetilde{F}-\partial_{x}^{2}\widetilde{F}+V(x)\widetilde{F}=i\chi'(t)F.
\end{equation*}
Taking the FBI transformation on both sides of the above identity and then using~\eqref{equ:heatSchro}, we obtain 
\begin{equation*}
    \partial_{s}W+\partial_{x}^{2}W-V(x)W=G,
\end{equation*}
where $W=\mathcal{T}_{h}\widetilde{F}$ and $G=-\mathcal{T}_{h}(i\chi'F)$. From Corollary~\ref{coro:1Dheatback}, there exists a constant $C=C(V,\zeta)$, depending only on $V$ and $\zeta$ such that for any $\tau\in \R$, we have 
\begin{equation}\label{est:W}
\begin{aligned}
    \left\|W(\tau)\right\|_{L^{2}(\R)}^{2}
    &\le Ce^{\frac{C}{T}}\int_{0}^{T}\left\|HW(\tau+is)\right\|_{L^{2}(\Omega)}^{2}\d s\\
    &+Ce^{\frac{C}{T}} 
    \int_{0}^{T}\|G(\tau+is)\|^{2}_{L^{2}(\R)}\d s.
    \end{aligned}
\end{equation}

\textbf{Step 2.} $L_{s}^{1}L_{x}^{2}$ estimates on $HW$ and $G$. First, from the definition of FBI transformation and $W(\tau+is)$, we have 
\begin{equation*}
    HW(\tau+is)=\frac{2^{\frac{1}{4}}}{(2\pi h)^{\frac{3}{4}}}e^{\frac{s^{2}}{2h}}\int_{\R}e^{-\frac{(\tau+t)^{2}}{2h}}e^{-i\frac{s(\tau+t)}{h}}\chi(t)HF(t)\d t.
\end{equation*}
It follows from Cauchy-Schwarz inequality that 
\begin{equation*}
    \left\|HW(\tau+is)\right\|_{L^{2}(\Omega)}^{2}\le \frac{20T}{(2\pi h)^{\frac{3}{2}}}e^{\frac{s^{2}}{h}}\int_{0}^{10T}\|HF(t)\|_{L^{2}(\Omega)}^{2}\d t.
\end{equation*}
Integrating the above inequality on $[0,T]$, we see that 
\begin{equation}\label{est:HW}
    \sup_{\tau\in \R}\int_{0}^{T}\left\|HW(\tau+is)\right\|_{L^{2}(\Omega)}^{2}\d s\le 
    \frac{20T^{2}}{(2\pi h)^{\frac{3}{2}}}e^{\frac{T^{2}}{h}}\int_{0}^{10T}\|HF(t)\|_{L^{2}(\Omega)}^{2}\d t.   \end{equation}
    Second, using again the definition of FBI transformation, 
    \begin{equation*}
        G(\tau+is)=-i\frac{2^{\frac{1}{4}}}{(2\pi h)^{\frac{3}{4}}}e^{\frac{s^{2}}{2h}}\int_{\R}e^{-\frac{(\tau+t)^{2}}{2h}}e^{-i\frac{s(\tau+t)}{h}}\chi'(t)F(t)\d t.   
        \end{equation*}
        Note that, from the definition of $\chi$ in~\eqref{def:chi}, we infer that 
        \begin{equation*}
            |\tau+t|\ge 2T,\ \  \mbox{for any}\ (\tau,t)\in [-6T,-4T]\times \mbox{supp}\chi'.
        \end{equation*}
        It follows from $\|F(t)\|_{L^{2}(\R)}=\|e^{itH}f\|_{L^{2}(\R)}=\|f\|_{L^{2}(\R)}$ that 
\begin{equation*}
    \max_{\tau\in [-6T,-4T]}\|G(\tau+is)\|_{L^{2}(\R)}^{2}\le \frac{128}{(2\pi h)^{\frac{3}{2}}}e^{\frac{s^{2}}{h}}e^{-\frac{4T^{2}}{h}}\|f\|_{L^{2}(\R)}^{2}.
\end{equation*}
Integrating the above inequality on $[0,T]$, we see that 
\begin{equation}\label{est:G}
    \max_{\tau \in [-6T,-4T]}\int_{0}^{T}\|G(\tau+is)\|_{L^{2}(\R)}^{2}\d s \le \frac{128T}{(2\pi h)^{\frac{3}{2}}}e^{-\frac{3T^{2}}{h}}\|f\|_{L^{2}(\R)}^{2}.
    \end{equation}
        
    \textbf{Step 3.} Conclusion. Combining~\eqref{est:W} and~\eqref{est:G} with~\eqref{est:HW}, we obtain
    \begin{equation}\label{est:W2}
    \begin{aligned}
\max_{\tau \in [-6T,-4T]}\left\|W(\tau)\right\|_{L^{2}(\R)}^{2}
&\le\frac{128CT}{(2\pi h)^{\frac{3}{2}}}e^{-\frac{3T^{2}}{h}+\frac{C}{T}}\|f\|_{L^{2}(\R)}^{2}\\
&+\frac{20CT^{2}}{(2\pi h)^{\frac{3}{2}}}e^{\frac{T^{2}}{h}+\frac{C}{T}}\int_{0}^{10T}\|HF(t)\|_{L^{2}(\Omega)}^{2}\d t.
\end{aligned}
\end{equation}
On the other hand, using again $\|F(t)\|_{L^{2}(\R)}=\|e^{itH}f\|_{L^{2}(\R)}=\|f\|_{L^{2}(\R)}$,
\begin{equation*}
    \|f\|_{L^{2}(\R)}^{2}=\frac{1}{2T}\int_{4T}^{6T}\|F(\tau)\|_{L^{2}(\R)}^{2}\d \tau\le I_{1}+I_{2},
\end{equation*}
where 
\begin{equation*}
\begin{aligned}
    I_{1}&=\frac{1}{T}\int_{4T}^{6T}\left\|\frac{1}{\sqrt{2\pi h}}\int_{\R}e^{-\frac{(\tau-t)^{2}}{2h}}\widetilde{F}(t)\d t\right\|_{L^{2}(\R)}^{2}\d \tau,\\
    I_{2}&= \frac{1}{T}\int_{4T}^{6T}\left\|F(\tau)
    -\frac{1}{\sqrt{2\pi h}}\int_{\R}e^{-\frac{(\tau-t)^{2}}{2h}}\widetilde{F}(t)\d t    \right\|_{L^{2}(\R)}^{2}\d \tau.
    \end{aligned}
    \end{equation*}
    Based on the definition of FBI transformation and $W(\tau)$, we have 
    \begin{equation*}
        I_{1}=\frac{1}{T}\int_{-6T}^{-4T}\left\|\frac{1}{\sqrt{2\pi h}}\int_{\R}e^{-\frac{(\tau+t)^{2}}{2h}}\widetilde{F}(t)\d t\right\|_{L^{2}(\R)}^{2}\d \tau=\frac{\sqrt{\pi h}}{T}\int_{-6T}^{-4T}\|W(\tau)\|_{L^{2}(\R)}^{2}\d \tau.
        \end{equation*}
        It follows from~\eqref{est:W2} that 
        \begin{equation}\label{est:I1}
        \begin{aligned}
            I_{1}
            &\le \frac{128CT}{{\pi h}}e^{-\frac{3T^{2}}{h}+\frac{C}{T}}\|f\|_{L^{2}(\R)}^{2}\\     &+\frac{20CT^{2}}{\pi h}e^{\frac{T^{2}}{h}+\frac{C}{T}}          \int_{0}^{10T}\left\|HF(t)\right\|_{L^{2}(\Omega)}^{2}\d t.
            \end{aligned}
        \end{equation}
        Next, using the fact that $\int_{\R}e^{-x^{2}}\d x=\sqrt{\pi}$ and the definition of $\chi$, we rewrite $I_{2}$ as 
        \begin{equation*}
            I_{2}=\frac{1}{T}\int_{4T}^{6T}\left\|
            \frac{1}{\sqrt{2\pi h}}\int_{\R}e^{-\frac{t^{2}}{2h}}\left(F(\tau)-\chi(\tau-t)F(\tau-t)\right)\d t    \right\|_{L^{2}(\R)}^{2}\d \tau=I_{21}+I_{22},        
    \end{equation*}
    where 
    \begin{equation*}
        \begin{aligned}
            I_{21}&=\frac{1}{T}\int_{4T}^{6T}\left\|
            \frac{1}{\sqrt{2\pi h}}\int_{|t|\ge 6T}e^{-\frac{t^{2}}{2h}}F(\tau)\d t    \right\|_{L^{2}(\R)}^{2}\d \tau,\\
            I_{22}&=\frac{1}{T}\int_{4T}^{6T}\left\|
            \frac{1}{\sqrt{2\pi h}}\int_{-6T}^{6T}e^{-\frac{t^{2}}{2h}}\left(F(\tau)-\chi(\tau-t)F(\tau-t)\right)\d t    \right\|_{L^{2}(\R)}^{2}\d \tau.
            \end{aligned}
    \end{equation*}
    Using again $\|F(t)\|_{L^{2}(\R)}=\|e^{itH}f\|_{L^{2}(\R)}=\|f\|_{L^{2}(\R)}$, we have 
    \begin{equation*}
        I_{21}\le \frac{1}{\pi h}\left(\int_{|t|\ge 6T}e^{-\frac{t^{2}}{2h}}\d t\right)^{2}\|f\|_{L^{2}(\R)}^{2}
        \le 2 e^{-\frac{18T^{2}}{h}}\|f\|_{L^{2}(\R)}^{2}.
    \end{equation*}
    Note that, from $f\in H^{2}(\R)$ and $H\ge \rm{Id}$, we have 
    \begin{equation*}
    \begin{aligned}
        \|F(t)\|_{L^{2}(\R)}&\le \|HF(t)\|_{L^{2}(\R)}=\|Hf\|_{L^{2}(\R)},\\
         \|\pt F(t)\|_{L^{2}(\R)}&=\|HF(t)\|_{L^{2}(\R)}=\|Hf\|_{L^{2}(\R)}.
         \end{aligned}
        \end{equation*}
        Therefore, from the mean-value theorem, for any $\tau\in (4T,6T)$, we have 
        \begin{equation*}
        \begin{aligned}
        &\left\|F(\tau)-\chi(\tau-t)F(\tau-t)\right\|_{L^{2}(\R)}\\
        &\le \|F(\tau)-F(\tau-t)\|_{L^{2}(\R)}+|\chi(\tau)-\chi(\tau-t)|\|F(\tau-t)\|_{L^{2}(\R)}\\
        &\le |t|\left(\|Hf\|_{L^{2}(\R)}+\|\chi'\|_{L^{\infty}(\R)}\|f\|_{L^{2}(\R)}\right)\le |t|\|Hf\|_{L^2(\R)} +\frac{2|t|}{T}\|f\|_{L^2(\R)}  .
        \end{aligned}
        \end{equation*}
        Based on the above inequality and Minkowski inequality, we directly have 
        \begin{equation*}
        	\begin{split}
 I_{22}\le &
   \frac{4}{\pi h}\left(\|Hf\|_{L^2(\R)}+\frac{2}{T}\|f\|_{L^2(\R)} \right)^{2}\left(\int_{0}^{6T}e^{-\frac{t^{2}}{2h}}t\d t\right)^{2} 
		\\ \le 
            &
            \frac{32h}{\pi}\left(\|Hf\|_{L^2(\R)}^2+\frac{1}{T^{2}}\|f\|_{L^2(\R)}^2\right).
            \end{split}
            \end{equation*}
            By the above estimates for $I_{21}$ and $I_{22}$, we see that 
            \begin{equation}\label{est:I2}
                I_{2}\le I_{21}+I_{22}\le \left(2e^{-\frac{18T^{2}}{h}}+
                		\frac{32h}{\pi T^2}
                		\right )\|f\|_{L^{2}(\R)}^{2}+\frac{32h}{\pi}\|Hf\|_{L^{2}(\R)}^{2}. 
                \end{equation}
                Gathering~\eqref{est:I1} and~\eqref{est:I2} together, we conclude that 
                \begin{equation*}
                \begin{aligned}
                \|f\|_{L^{2}(\R)}^{2}
                &\le \frac{32h}{\pi}\|Hf\|_{L^{2}(\R)}^{2} \\                
                &+
                \frac{20CT^{2}}{{\pi h}}e^{\frac{T^{2}}{h}+\frac{C}{T}}\int_{0}^{10T}\left\|HF(t)\right\|_{L^{2}(\Omega)}^{2}\d t\\
                &+\left(\frac{128CT}{{\pi h}}e^{-\frac{3T^{2}}{h}+\frac{C}{T}}+2e^{-\frac{18T^{2}}{h}}+\frac{32h}{\pi T^2} \right)\|f\|_{L^{2}(\R)}^{2}.
                       \end{aligned}
                \end{equation*}
                	 				 By taking $h_{0}$ small enough (independent of $T>0$), we deduce that for all $h$  satisfying $$0<h<h_{0}(1+T^{-3})^{-1},$$
                	 				 we have
                 \begin{align*}
                  \frac{128CT}{\pi h}e^{-\frac{3T^2}{h}+\frac{C}{T}}+2e^{-\frac{18T^2}{h}}+\frac{32h}{\pi T^2}
                 <\frac{1}{2}.
                   \end{align*}
                 This completes the proof of Proposition~\ref{prop:key}.
\end{proof}

We are in a position to complete the proof of Theorem~\ref{thm:main1}.

\begin{proof}[End of the proof of Theorem~\ref{thm:main1}]
Recall that, without loss of generality, we assume $V\in C(\R)\cap L^{\infty}(\R)$ with $V\ge 1$. It follows that the operator $H\ge {\rm{{Id}}}$ and thus $H$ is invertible and is a bijection from $H^{2}(\R)$ to $L^{2}(\R)$. For any $u_{0}\in L^{2}(\R)$, we denote $U_{0}=H^{-1}u_{0}\in H^{2}(\R)$ and $u(t)=e^{itH}u_{0}$ be the corresponding solution of~\eqref{equ:LS}. 

First, from Proposition~\ref{prop:key}, there exist some constants $h_{0}=h_{0}(V,\zeta)\in (0,1)$ and $C=C(V,\zeta)>0$ depending only on $V$ and $\zeta$, such that for any $T>0$ and any $0<h<h_{0}(1+T^{-3})^{-1}$, we have 
\begin{equation}\label{est:U0}
\begin{aligned}
    \|U_{0}\|_{L^{2}(\R)}^{2}
    &\le Ch\|u_{0}\|_{L^{2}(\R)}^{2}
    +Ce^{\frac{2T^{2}}{h}+\frac{C}{T}}\int_{0}^{T}\|u(t)\|_{L^{2}(\Omega)}^{2}\d t.
\end{aligned}
\end{equation}
Second, using Corollary~\ref{coro:High} and the triangle inequality, there exists $\mu_{1}=\mu_{1}(V,\zeta)>0$ and $C_{1}=C_{1}(V,\zeta)>0$, depending only on $V$ and $\zeta$, such that for any $T>0$ and $\mu>\mu_{1}(1+T^{-1}) $, we have 
\begin{equation*}
    \|u_{0}\|_{L^{2}(\R)}^{2}\le \frac{C_{1}}{T}\int_{0}^{T}\|u(t)\|_{L^{2}(\Omega)}^{2}\d t+(C_{1}+1)\|\Pi_{\mu} u_{0}\|_{L^{2}(\R)}^{2},
\end{equation*}
 Note that, from the definition of $U_{0}$ and $\Pi_{\mu}$, we find 
\begin{equation*}
    \|\Pi_{\mu} u_{0}\|_{L^{2}(\R)}^{2}=\int_{0}^{\mu}(\d m_{\lambda}u_{0},u_{0})_{L^{2}(\R)}=\int_{0}^{\mu}\lambda^{4}\left(\d m_{\lambda} U_{0},U_{0}\right)_{L^{2}(\R)}\le \mu^{4}\left\|U_{0}\right\|_{L^{2}(\R)}^{2},
    \end{equation*}
    which implies 
        \begin{equation}\label{est:u0}
    \|u_{0}\|_{L^{2}(\R)}^{2}\le \frac{C_{1}}{T}\int_{0}^{T}\|u(t)\|_{L^{2}(\Omega)}^{2}\d t+(C_{1}+1)\mu^{4}\|U_{0}\|_{L^{2}(\R)}^{2}.
   \end{equation}
   Combining~\eqref{est:U0} and~\eqref{est:u0}, we conclude that 
   \begin{equation*}
   \begin{aligned}
\|u_{0}\|_{L^{2}(\R)}^{2}
&\le C(C_{1}+1)\mu^{4}h\|u_{0}\|_{L^{2}(\R)}^{2}\\
&+\left(C(C_{1}+1)\mu^{4}e^{\frac{2T^{2}}{h}+\frac{C}{T}}+\frac{C_{1}}{T}\right)\int_{0}^{T}\|u(t)\|_{L^{2}(\Omega)}^{2}\d t.
\end{aligned}
\end{equation*}
This completes the proof of Theorem~\ref{thm:main1} for taking $\mu=2\mu_{1}(1+T^{-1}) $ and $h=\varepsilon(1+T^{-4})^{-1}$ where $\varepsilon$ small enough. The proof of Theorem \ref{thm:main1} is complete.
\end{proof}

\begin{remark}\label{remark:infty}
We now briefly sketch the proof of Theorem~\ref{thm:main1} and Corollary~\ref{cor:control} to the case of the potential $V\in L^{\infty}(\R)$.
Actually, for any potential $V\in L^{\infty}(\R)$, there exists a sequence of continue potentials $\left\{V_{n}\right\}_{n=1}^{\infty}\subset C(\R)\cap L^{\infty}(\R)$ such that 
\begin{equation*}
\lim_{n\to \infty}V_{n}(x)=V(x),\ \ \mbox{a.e.}\ x\in \R \quad \mbox{and}\quad \sup_{n\in \mathbb{N^{+}}}\|V_{n}\|_{{\infty}}\le \|V\|_{{\infty}}.
\end{equation*}
On the other hand, if we trace the dependence of the constant $C$ which appears in Theorem~\ref{thm:main1} more carefully, we could find the constant depends only on $\|V\|_{\infty}, L$ and $\zeta$ (see for instance Lemma~\ref{le:spectra} and Lemma~\ref{le:resolvent}, and~\cite[Proposition 4.5]{WANG}). Therefore, based on Theorem~\ref{thm:main1}, there exists a constant $C=C(\|V\|_{\infty},L,\zeta)>0$ depending only on $\|V\|_{\infty},L,$ and $\zeta$ such that, for any $n\in \mathbb{N}^{+}$, we have 
\begin{equation*}
\|u_{0}\|_{L^{2}(\R)}^{2}\le Ce^{\frac{C}{T^{2}}} \int_{0}^{T}\|u_{n}(t)\|_{L^{2}(\Omega)}^{2}\d t,
\end{equation*}
where $u_{n}$ is the solution for the following 1D Schr\"odinger equation,
\begin{equation*}
i\partial_{t}u_{n}-\partial_{x}^{2}u_{n}+V_{n}(x)u_{n}=0,\quad {u_{n}}_{|t=0}=u_{0}\in L^{2}(\R).
\end{equation*}
Recall that, we denote by $u$ the solution of the corresponding solution of~\eqref{equ:LS} with initial data $u_{0}\in L^{2}(\R)$. It follows from the triangle inequality that
\begin{equation*}
\begin{aligned}
\|u_{0}\|_{L^{2}(\R)}^{2}\le 
&Ce^{\frac{C}{T^{2}}} \int_{0}^{T}\|u(t)\|_{L^{2}(\Omega)}^{2}\d t
+Ce^{\frac{C}{T^{2}}} \int_{0}^{T}\|u(t)-u_{n}(t)\|_{L^{2}(\R)}^{2}\d t\\
\le &Ce^{\frac{C}{T^{2}}} \int_{0}^{T}\|u(t)\|_{L^{2}(\Omega)}^{2}\d t+Ce^{\frac{C}{T^{2}}} 
\int_{0}^{T}\int_{0}^{t}\|(V-V_{n})u(s)\|^{2}_{L^{2}(\R)}\d s\d t.
\end{aligned}
\end{equation*}
Last, from the dominated convergence Theorem and the classical Hilbert uniqueness method, we complete the proof of Theorem~\ref{thm:main1} and Corollary~\ref{cor:control} to the case of the potential $V\in L^{\infty}(\R)$.
\end{remark}

\appendix
\section{Proof of Theorem~\ref{thm:ZHU}}\label{AppA}

In this appendix, we repeat the proof of Theorem~\ref{thm:ZHU} in~\cite{ZHU} based on complex analysis. First, we recall the following Jensen's formula from~\cite[Theorem 15.18]{Rudin}.

\begin{lemma}[Jensen's formula]
Let $0<r_{1}<r_{2}<\infty$. Let $f$ be a holomorphic function in ${B}_{r_{2}}$ with $f(0)\ne 0$ and $a_{1}, a_{2},\dots,a_{N}$ are the zeros of $f$ in $B_{r_{1}}$ $($repeated according to their respective multiplicity$)$, then we have 
\begin{equation*}
    \log |f(0)|=\sum_{n=1}^{N}\log \left(\frac{|a_{n}|}{r_{1}}\right)+\frac{1}{2\pi }\int_{0}^{2\pi}\left|f(r_{1}e^{i\theta})\right|\d \theta.
\end{equation*}
    
\end{lemma}

Note that, Jensen's formula can be used to estimate the number of zeros of the holomorphic function in an open disk.
\begin{corollary}\label{coro:zero}
Let $0<r_{1}<r_{2}<r_{3}<\infty$ and $a\in \mathbb{C}$. Let $f$ be a holomorphic function in ${B}_{r_{3}}(a)$ with $f(a)\ne 0$ and $a_{1}, a_{2},\dots,a_{N}$ are the zeros of $f$ in $B_{r_{1}}(a)$ $($repeated according to their respective multiplicity$)$, then we have   
\begin{equation*}
    N\le \frac{\log M-\log |f(a)|}{\log r_{2}-\log r_{1}},\quad \mbox{where}\ M=\max_{|z|=r_{2}}|f(z)|.
\end{equation*}
\end{corollary}

Second, we recall Hadamard's three-circle theorem from~\cite[Page 264]{Rudin}.
\begin{lemma}\label{le:Hada}
    Let $f$ be a holomorphic function in $B_{R}$ for $0<R<\infty$.
    Let $M(r)$ be the maximum of $|f(z)|$ on the circle $|z|=r$ for $0<r<R$. Then we have 
   \begin{equation*}
       \log\left(\frac{r_{3}}{r_{1}}\right)\log M(r_{2})\le 
       \log\left(\frac{r_{3}}{r_{2}}\right)\log M(r_{1})
       +\log\left(\frac{r_{2}}{r_{1}}\right)\log M(r_{3}),
       \end{equation*}
       for any three concentric circles of radii $0<r_{1}<r_{2}<r_{3}<R$.
 \end{lemma}
Third, following from~\cite[Theorem 4.3]{FY}, we introduce the Remez-type inequality for holomorphic polynomials for further reference.

\begin{lemma}[\cite{FY}]\label{le:poly}
Let $P(z)$ be a holomorphic polynomial of degree $N$. Let $E\subset B_{1}$. Then for any $\delta>0$, we have 
\begin{equation*}
    \sup_{B_{1}}\left|P(z)\right|\le \left(\frac{6e}{\mathcal{H}_{\delta}(E)}\right)^{\frac{N}{\delta}}\sup_{E}\left|P(z)\right|.
\end{equation*}
    Here, $\mathcal{H}_{\delta}(E)$ means the $\delta$-dimensional Hausdorff content of $E$.
\end{lemma}
    
Now we recall some elementary properties of quasiconformal mapping and the presentation is close to~\cite{AIM}. For complex function $f$, we write the derivative as 
\begin{equation*}
    Df(z)h=\frac{\partial f}{\partial z}(z)h+\frac{\partial f}{\partial \bar{z}}(z)\overline{h},\quad \mbox{for}\ h\in \mathbb{C}.
\end{equation*}
The norm of the derivative and Jacobian can be explained as 
\begin{equation*}
    |Df(z)|=\left|\frac{\partial f}{\partial z}(z)\right|+\left|\frac{\partial f}{\partial \bar{z}}(z)\right|\quad \mbox{and}\quad 
    Jf(z)=\left|\frac{\partial f}{\partial z}(z)\right|^{2}-\left|\frac{\partial f}{\partial \bar{z}}(z)\right|^{2}.
    \end{equation*}

\begin{definition}
    Let $U$ and $V$ be open sets of $\mathbb{C}$ and take $K\ge 1$.

(i) An orientation-preserving mapping $f:U\to V$ is $K$-quasiregular mapping if 
    \begin{equation*}
    f\in W_{loc}^{1,2}(U)\quad \mbox{and}\quad 
    \left|Df(z)\right|^{2}\le K Jf(z),\quad \mbox{for almost every}\ z\in U.
    \end{equation*}    
    
    (ii) An orientation-preserving homeomorphism $f:U\to V$ is $K$-quasiconformal if 
    \begin{equation*}
    f\in W_{loc}^{1,2}(U)\quad \mbox{and}\quad 
    \left|Df(z)\right|^{2}\le K Jf(z),\quad \mbox{for almost every}\ z\in U.
    \end{equation*}
    
    \end{definition}

Following~\cite[Chapter 16]{AIM}, we denote 
\begin{equation*}
    *=\begin{pmatrix}
        0 &-1\\
        1 &0
    \end{pmatrix}:\R^{2}\to \R^{2}\quad \mbox{and}\ **=-{\rm{Id}}.
    \end{equation*}

    For any solution $\phi$ of~\eqref{equ:2Dell}, we find  the field $(*A\nabla \phi)$ is curl-free, and thus from Poincar\'e lemma, there exists a Sobolev function $\psi\in W_{loc}^{1,2}(B_{4})$ such that 
    \begin{equation*}
        \nabla \psi= *A\nabla \phi=\begin{pmatrix}
            0 &-1\\
            1 &0
        \end{pmatrix}A\nabla \phi.
        \end{equation*}
    Note that the function $\psi$ is unique up to an additive constant and it is called the $A$-harmonic conjugate of $\phi$.
    For any $a\in B_{1}$,
    we consider the complex function $f_{a}=\phi+i\psi_{a}$ where $\psi_{a}$ is $A$-harmonic conjugate with $\psi_{a}(a)=0$. By an elementary computation and the definition of $\psi_{a}$
    \begin{equation*}
        |Df_{a}|^{2}\le |\nabla \phi|^{2}+|\nabla \psi_{a}|^{2}\le \left(\Lambda+\Lambda^{-1}\right)\nabla \phi\cdot A\nabla \phi
        \le \left(\Lambda+\Lambda^{-1}\right)\nabla \phi\cdot (-*\nabla \psi_{a}).    \end{equation*}
        Note that, from the definition of $*$, we find
        \begin{equation*}
        \nabla \phi\cdot (-*\nabla \psi_{a})=(\partial_{x}\phi)(\partial_{y}\psi_{a})-(\partial_{y}\phi)(\partial_{x}\psi_{a})=
        Jf_{a}.
        \end{equation*}
        Based on the above inequalities, we obtain
        \begin{equation*}
            |Df_{a}(z)|^{2}\le (\Lambda+\Lambda^{-1})Jf_{a}(z),\quad \mbox{for almost every}\ z\in B_{4},
        \end{equation*}
        which means that $f_{a}$ is a $(\Lambda+\Lambda^{-1})$-quasiregular mapping. It then follows from the representation theorem (see~\cite[Section 2]{AEESAIM} and~\cite[Corollary 5.5.3]{AIM}) that $f_{a}$ can be written as
        \begin{equation*}
            f_{a}=F\circ G,\quad \mbox{on}\ B_{2},
        \end{equation*}
        where $F$ is holomorphic in $B_{2}$, and $G$ is a $(\Lambda+\Lambda^{-1})$-quasiconformal homeomorphism from $B_{2}$ onto $B_{2}$ which verifies $G(0)=0$ and 
        \begin{equation}\label{est:Holder}
            C^{-1}|z_{2}-z_{1}|^{\frac{1}{\alpha}}\le 
            |G(z_{2})-G({z_{1}})|\le C|z_{2}-z_{1}|^{\alpha},\ \ \mbox{when}\ (z_{1},z_{2})\in B_{2}\times B_{2},
        \end{equation}
        for some constants $\alpha=\alpha(\Lambda)\in (0,1)$ and $C=C(\Lambda)>0$ depending only on $\Lambda$. Moreover, from~\cite[Corollary 5.9.2]{AIM}, we know that $G$ has a continuous homeomorphic extension to the boundary $\partial B_{2}$. Therefore, from~\eqref{est:Holder}, there exists a constant $r=r(\Lambda)\in (0,2)$, depending only $\Lambda$, such that 
        \begin{equation}\label{est:dist}
          0<2-r\le   {\rm{dist}}\left(G(B_{1}),\partial B_{2}\right)\Longrightarrow G(B_{1})\subset B_{r}.
        \end{equation}
        For any $a\in B_{1}$, we consider the holomorphic self-homeomorphism of $B_{2}$ to itself:
        \begin{equation*}
            R_{a}(z)=4\frac{z-G(a)}{4-\overline{G(a)}z}: B_{2}\to B_{2}.
        \end{equation*}
        Based on~\cite[Theorem 3.1.2]{AIM}, we know that $f_{a}$ can be rewritten as 
        \begin{equation}\label{equ:fFaGa}
            f_{a}=\left(F\circ R_{a}^{-1}\right)\circ (R_{a}\circ G):=F_{a}\circ G_{a},\quad \mbox{on}\ B_{2},
        \end{equation}
        where $F_{a}$ is holomorphic in $B_{2}$, and $G_{a}$ is a $(\Lambda+\Lambda^{-1})$-quasiconformal homeomorphism from $B_{2}$ onto $B_{2}$ which verifies $G_{a}(a)=0$.       Moreover, from~\eqref{est:Holder} and~\eqref{est:dist}, 
        \begin{equation}\label{est:Ga}
            C^{-1}|z_{2}-z_{1}|^{\frac{1}{\alpha}}\le 
            |G_{a}(z_{2})-G_{a}({z_{1}})|\le C|z_{2}-z_{1}|^{\alpha},\ \ \mbox{when}\ (z_{1},z_{2})\in B_{2}\times B_{2},
        \end{equation}
        where $\alpha=\alpha(\Lambda)\in (0,1)$ and $C=C(\Lambda)>0$ are constants depending only on $\Lambda$. For any $a\in B_{1}$, using~\eqref{est:Ga}, we know that~\eqref{est:dist} is also true for $G_{a}$ with probably different constant $r=r(\Lambda)$ but it still depends only on $\Lambda$.

The following proposition plays a crucial role in our proof of Theorem~\ref{thm:ZHU}.

\begin{proposition}\label{prop:ZHU}
    Let $\omega\subset B_{1}\cap \ell_{0}$ satisfy $|\omega|>0$ for some line $\ell_{0}\in \mathbb{R}^{2}$. Then there exist $z_{0}\in \omega$ and some constants $\alpha=\alpha(\Lambda,|\omega|)\in (0,1)$ and
    $C=C(\Lambda,|\omega|)>0$,    
    depending only on $\Lambda$ and $|\omega|$, such that for any $H^{1}_{loc}$ solution $\phi$ of~\eqref{equ:2Dell} with its $A$-harmonic conjugate satisfying $\psi_{z_{0}}(z_{0})=0$, we have 
    \begin{equation*}
        \sup_{B_{1}}|\phi|\le C\left(\sup_{\omega}|\phi+i\psi_{z_{0}}|^{\alpha}\right)\left(\sup_{B_{2}}|\phi|^{1-\alpha}\right).
    \end{equation*}
\end{proposition}

\begin{proof}
From $\omega \subset B_{1}\cap \ell_{0}$ and~\eqref{est:Ga}, there exist $r_{1}=r_{1}(|\omega|)\in (0,1)$ depending only on $|\omega|$ and $r_{2}=r_{2}(\Lambda,|\omega|)\in (0,\frac{1}{6})$ depending only on $\Lambda$ and $|\omega|$ such that 
\begin{equation}\label{est:12omega}
    \frac{1}{2}|\omega|\le \left|\omega\cap B_{r_{1}}\right|\quad \mbox{and}\quad 
    B_{6r_{2}}\subset \bigcap_{a\in B_{r_{1}}}G_{a}\left(B_{1}\right).
\end{equation}
Moreover, from~\eqref{est:Ga} and~\eqref{est:12omega}, there exist $z_{0}\in \omega\cap B_{r_{1}}$ and $r_{3}=r_{3}(\Lambda,|\omega|)\in (0,1)$ depending only on $\Lambda$ and $|\omega|$ such that 
\begin{equation}\label{est:omega}
    c|\omega|\le |\omega\cap B_{r_{3}}(z_{0})|\quad \mbox{and}\quad 
    G_{z_{0}}\left(B_{r_{3}}(z_{0})\right)\subset B_{r_{2}},
\end{equation}
where $c=c(\Lambda,|\omega|)$ depends only on $\Lambda$ and $|\omega|$. We denote $\widetilde{\omega}=\omega\cap B_{r_{3}}(z_{0})$.

Let $a_{1},a_{2},\dots,a_{N}$ be the zeros of $F_{z_{0}}$ in $B_{2r_{2}}$
$($repeated according to their respective multiplicity$)$. Let $a_{0}\in \partial B_{r_{2}}$ be such that $|F_{z_{0}}(a_{0})|$ is the maximum of $|F_{z_{0}}(z)|$
 on $B_{r_{2}}$.
We consider the following complex polynomial $P_{z_{0}}(z)$ and holomorphic non-vanishing function $h_{z_{0}}(z)$ on $B_{6r_{2}}$,
\begin{equation*}
    P_{z_{0}}(z)=\prod_{n=1}^{N}(z-a_{n})\quad \mbox{and}\quad 
    h_{z_{0}}(z)=\frac{F_{z_{0}}(z)}{P_{z_{0}}(z)}.
\end{equation*}
On the one hand, we denote by $\widetilde{N}$ the number of zeros of $F_{z_{0}}$ in $B_{3r_{2}}(a_{0})$. From $B_{2r_{2}}\subset B_{3r_{2}}(a_{0})$,
$B_{\frac{7}{2}r_{2}}(a_{0})\subset B_{5r_{2}}$ and 
Corollary~\ref{coro:zero},
\begin{equation}\label{est:N}
\begin{aligned}
    N\le \widetilde{N}
    &\le \frac{1}{\log \frac{7}{6}}\left(\log \sup_{B_{\frac{7}{2}r_{2}}(a_{0})}|F_{z_{0}}|-\log |F_{z_{0}}(a_{0})|\right)\\
    &\le \frac{1}{\log \frac{7}{6}}
    \left(\log \sup_{B_{5r_{2}}}|F_{z_{0}}|-\log\sup_{B_{r_{2}}} |F_{z_{0}}|\right).
    \end{aligned}
\end{equation}

Then, using the definition of $h_{z_{0}}$ and $P_{z_{0}}$, we see that  
\begin{equation}\label{est:F1}
    \sup_{B_{r_{2}}}|F_{z_{0}}|=|F_{z_{0}}(a_{0})|\le |h_{z_{0}}(a_{0})||P_{z_{0}}(a_{0})|\le (3r_{2})^{N}|h_{z_{0}}(a_{0})|. 
    \end{equation}
    Using the maximum modulus principle on $B_{5r_{2}}$, we find
    \begin{equation}\label{est:h1}
    \sup_{B_{5r_{2}}}|h_{z_{0}}|\le \left(\sup_{\partial{B_{5r_{2}}}}|F_{z_{0}}|\right)\left(\sup_{\partial{B_{5r_{2}}}}|P_{z_{0}}^{-1}|\right)\le \left(\frac{1}{3r_{2}}\right)^{N}   \sup_{{B_{5r_{2}}}}|F_{z_{0}}|.
\end{equation}
Since $h_{z_{0}}$ is a holomorphic non-vanishing function, $\log |h_{z_{0}}|$ is a harmonic function on $B_{2r_{2}}$. From Harnack's inequality for harmonic function, there exists a constant $C=C(\Lambda)>1$, depending only on $\Lambda$, such that 
\begin{equation*}
    \sup_{B_{r_{2}}}\left(\sup_{B_{5r_{2}}}\log |h_{z_{0}}|-\log |h_{z_{0}}(z)|\right)\le 
    C\inf_{B_{r_{2}}}\left(\sup_{B_{5r_{2}}}\log |h_{z_{0}}|-\log |h_{z_{0}}(z)|\right).
\end{equation*}
It follows that 
\begin{equation*}
    |h_{z_{0}}(a_{0})|^{C}\sup_{B_{5r_{2}}}|h_{z_{0}}|\le 
    \left(\sup_{B_{5r_{2}}}|h_{z_{0}}|^{C}\right)
    \left(\inf_{B_{r_{2}}}|h_{z_{0}}|\right).
\end{equation*}
Combining~\eqref{est:F1} and~\eqref{est:h1} with the above inequality, we obtain
\begin{equation}\label{est:F}
\left(\sup_{B_{r_{2}}}|F_{z_{0}}|^{C}\right)\left(\sup_{B_{5r_{2}}}|h_{z_{0}}|\right)\le \left(\sup_{B_{5r_{2}}}|F_{z_{0}}|^{C}\right)\left(\inf_{B_{r_{2}}}|h_{z_{0}}|\right).
\end{equation}
On the other hand, using Lemma~\ref{le:poly}, we find 
\begin{equation*}
    \sup_{B_{r_{2}}}|P_{z_{0}}|\le \left(\frac{6e}{\mathcal{H}_{\alpha}(G_{z_{0}}(\widetilde{\omega}))}\right)^{\frac{N}{\alpha}}\sup_{G_{z_{0}}(\widetilde{\omega})}|P_{z_{0}}|.
\end{equation*}
From~\eqref{equ:measu},~\eqref{est:Ga} and~\eqref{est:omega}, there exists a constant $C_{1}=C_{1}(\Lambda)>0$, depending only on $\Lambda$, such that
\begin{equation}\label{est:P}
  c|\omega|\le  |\widetilde{\omega}|\le C_{1}\mathcal{H}_{\alpha}(G_{z_{0}}(\widetilde{\omega}))\Longrightarrow
     \sup_{B_{r_{2}}}|P_{z_{0}}|\le \left(\frac{6C_{1}e}{c|{\omega}|}\right)^{\frac{N}{\alpha}}\sup_{G_{z_{0}}(\widetilde{\omega})}|P_{z_{0}}|.
     \end{equation}
     Thanks to~\eqref{est:omega},~\eqref{est:F} and~\eqref{est:P}, we find
     \begin{equation*}
     \begin{aligned}
         \sup_{B_{r_{2}}}|F_{z_{0}}|^{1+C}
         &\le \left(\sup_{B_{r_{2}}}|F_{z_{0}}|^{C}\right)\left(\sup_{B_{5r_{2}}}|h_{z_{0}}|\right)\left(\sup_{B_{r_{2}}}|P_{z_{0}}|\right)\\
         &\le C^{N}_{2}\left(\sup_{B_{5r_{2}}}|F_{z_{0}}|^{C}\right)\left(\sup_{G_{z_{0}}({\widetilde{\omega}})}|P_{z_{0}}|\right)    \left(\inf_{B_{r_{2}}}|h_{z_{0}}|\right)\\
         &\le C^{N}_{2}\left(\sup_{B_{5r_{2}}}|F_{z_{0}}|^{C}\right)\left(\sup_{G_{z_{0}}({\widetilde{\omega}})}|F_{z_{0}}|\right),
         \end{aligned}
     \end{equation*}
     where $C_{2}=C_{2}(\Lambda,|\omega|)>1$ is a constant depending only on $\Lambda$ and $|\omega|$. Combining the above inequality with~\eqref{est:N}, there exists a constant $C_{3}=C_{3}(\Lambda,|\omega|)>0$, depending only on $\Lambda$ and $|\omega|$, such that 
     \begin{equation*}
         \sup_{B_{r_{2}}}|F_{z_{0}}|^{1+C+C_{3}}\le \left(\sup_{B_{5r_{2}}}|F_{z_{0}}|^{C+
         C_{3}  }\right)\left(\sup_{G_{z_{0}}({\widetilde{\omega}})}|F_{z_{0}}|\right),
           \end{equation*}
           and thus we have 
           \begin{equation*}
            \sup_{B_{r_{2}}}|F_{z_{0}}|\le \left(\sup_{B_{5r_{2}}}|F_{z_{0}}|^{1-\alpha}\right)
            \left(\sup_{G_{z_{0}}(\widetilde{\omega})}|F_{z_{0}}|^{\alpha}\right),\quad \mbox{where}\ \alpha=\frac{1}{1+C+C_{3}}\in (0,1).
           \end{equation*}
           Therefore, from~\eqref{est:dist} and Lemma~\ref{le:Hada}, there exists a constant $r=r(\Lambda)\in (0,2)$, depending only on $\Lambda$, such that  
           \begin{equation}\label{est:F11}
           \begin{aligned}
               \sup_{G_{z_{0}}(B_{1})}|F_{z_{0}}|\le \sup_{B_{r}}|F_{z_{0}}|
               &\le 
                \left(\sup_{B_{r_{2}}}|F_{z_{0}}|^{\alpha_{1}} \right)
                \left( \sup_{B_{\frac{2+r}{2}}}|F_{z_{0}}|^{1-\alpha_{1}}\right)\\
                &\le 
            \left(\sup_{G_{z_{0}}(\widetilde{\omega})}|F_{z_{0}}|^{\alpha\alpha_{1}}\right)
                \left(\sup_{B_{\frac{2+r}{2}}}|F_{z_{0}}|^{1-\alpha\alpha_{1}}\right),              \end{aligned}
                \end{equation}
                where $\alpha_{1}=\alpha_{1}(\Lambda,|\omega|)\in (0,1)$ is a constant depending only on $\Lambda$ and $|\omega|$.

        From the definition of $F_{z_{0}}$, the holomorphic function $F_{z_{0}}$ can be written as 
        \begin{equation}\label{equ:Fz0}
            F_{z_{0}}=f_{z_{0}}\circ G_{z_{0}}^{-1}=\phi\circ G_{z_{0}}^{-1}+i\psi_{z_{0}}\circ G_{z_{0}}^{-1}.
        \end{equation}
        Therefore, using the Cauchy-Riemann equation and $\psi_{z_{0}}\circ G_{z_{0}}^{-1}(0)=0$,
        \begin{equation*}
            \psi_{z_{0}}\circ G_{z_{0}}^{-1}(x,y)=\int_{0}^{1}(y,-x)\cdot\nabla(\phi\circ G_{z_{0}}^{-1})(x\sigma,y\sigma)\d \sigma,\ \mbox{on}\ B_{2},
        \end{equation*}
        and thus, from interior estimates of derivatives of harmonic function (see for instance~\cite[Theorem 2.10]{GN}), there exists a constant $C_{4}=C_{4}(\Lambda)$, depending only on $\Lambda$, such that 
        \begin{equation*}
            \sup_{B_{\frac{2+r}{2}}} \left|\psi_{z_{0}}\circ G_{z_{0}}^{-1} \right|
            \le 2\sup_{B_{\frac{2+r}{2}}} \left|\nabla (\phi\circ G_{z_{0}}^{-1}) \right|\le C_{4}\sup_{B_{\frac{6+r}{4}}}
            \left| \phi\circ G_{z_{0}}^{-1} \right| .          \end{equation*}
            Hence, using again~\eqref{equ:Fz0}, we find
            \begin{equation}\label{est:F2}
            \begin{aligned}
                \sup_{B_{\frac{2+r}{2}}}|F_{z_{0}}|
                &\le 
                \sup_{B_{\frac{2+r}{2}}}|\phi\circ G_{z_{0}}^{-1}|+
                 \sup_{B_{\frac{2+r}{2}}}|\psi_{z_{0}}\circ G_{z_{0}}^{-1}| \\
                & \le (1+C_{4})\sup_{B_{\frac{6+r}{4}}}
            \left| \phi\circ G_{z_{0}}^{-1} \right|
            \le (1+C_{4})\sup_{B_{2}}|\phi|.    
            \end{aligned}
            \end{equation}
            Combining~\eqref{est:F11} and~\eqref{est:F2} with $\widetilde{\omega}\subset \omega$, we conclude that
            \begin{equation*}
            \begin{aligned}
                \sup_{G_{z_{0}}(B_{1})}|F_{z_{0}}|
                &\le 
                (1+C_{4})^{1-\alpha\alpha_{1}}\left(\sup_{G_{z_{0}}(\widetilde{\omega})}|F_{z_{0}}|^{\alpha\alpha_{1}}\right)
                \left(\sup_{B_{2}}|\phi|^{1-\alpha\alpha_{1}}\right)\\
                &\le (1+C_{4})^{1-\alpha\alpha_{1}}\left(\sup_{G_{z_{0}}({\omega})}|F_{z_{0}}|^{\alpha\alpha_{1}}\right)
                \left(\sup_{B_{2}}|\phi|^{1-\alpha\alpha_{1}}\right),                \end{aligned}
            \end{equation*}
            which completes the proof of Proposition~\ref{prop:ZHU}
        \end{proof}

We are in a position to complete the proof of Theorem~\ref{thm:ZHU}.

\begin{proof}[End of the proof of Theorem~\ref{thm:ZHU}]
Let $z_{0}\in\omega$ be the point appearing in Proposition~\ref{prop:ZHU} and $\psi_{z_{0}}$ be the $A$-harmonic conjugate of $\phi$ satisfying $\psi_{z_{0}}(z_{0})=0$. Note that, from $A\nabla\phi\cdot\textbf{e}_{0}=0$ and the definition of $\psi_{z_{0}}$, we obtain
\begin{equation*}
    \nabla \psi_{z_{0}} \cdot \textbf{e}_{0}^{\perp}=
   *
    A\nabla \phi \cdot \textbf{e}_{0}^{\perp}=A\nabla \phi \cdot 
   (-*
    \textbf{e}_{0}^{\perp})=A\nabla \phi\cdot {\textbf{e}_{0}}=0.
    \end{equation*}
    It follows that $\psi_{z_{0}}(z)=\psi_{z_{0}}(z_{0})=0$ on $\ell_{0}$. Therefore, from $\omega\subset B_{1}\cap \ell_{0}$ and Proposition~\ref{prop:ZHU}, we conclude that, there exist some constants $\alpha=\alpha(\Lambda,|\omega|)\in (0,1)$ and $C=C(\Lambda,|\omega|)>0$, depending only on $\Lambda$ and $|\omega|$, such that
    \begin{equation*}
        \sup_{B_{1}}|\phi|\le C\left(\sup_{\omega}|\phi+i\psi_{z_{0}}|^{\alpha}\right)\left(\sup_{B_{2}}|\phi|^{1-\alpha}\right)\le C\left(\sup_{\omega}|\phi|^{\alpha}\right)\left(\sup_{B_{2}}|\phi|^{1-\alpha}\right).    \end{equation*}
        The proof of Theorem~\ref{thm:ZHU} is complete.
\end{proof}

\section{Proof of Proposition~\ref{prop:heat}}\label{AppB}

The implication of backward observability for the heat equation from the spectral inequality is standard. Here we closely follow the strategy of~\cite[Theorem 8 and Corollary 4.1]{BurqMoyanoJEMS} which builds upon the early work in~\cite{AEWZJEMS}. The key point is the following interpolation estimates for solutions of the heat equation.

\begin{lemma}\label{le:inter}
    Let $\Omega$ be a $(1,\zeta)$-thick set and let $V\in C(\R)\cap L^{\infty}(\R)$ with $V\ge 1$. Then there exists a constant $C=C(V,\zeta)>0$, depending only on $V$ and $\zeta$, such that for any $\alpha\in (0,1)$, $0\le s<t<\infty$ and $f\in L^{2}(\R)$, we have 
    \begin{equation*}
        \|e^{-tH}f\|_{L^{2}(\R)}\le Ce^{\frac{3}{\alpha(t-s)}}\|e^{-tH}f\|_{L^{2}(\Omega)}^{1-\alpha}\|e^{-sH}f\|_{L^{2}(\R)}^{\alpha}.
    \end{equation*}
\end{lemma}

\begin{proof}
First, for any $\mu>0$ and $0\le s<t<\infty$, we have 
\begin{equation}\label{est:tH}
\begin{aligned}
    \|e^{-tH}(1-\Pi_{\mu})f\|_{L^{2}(\R)}^{2}
    &=\int_{\mu}^{\infty}e^{-2t\lambda^{2}}(\d m_{\lambda}f,f)_{L^{2}(\R)}\\
   & \le e^{-2(t-s)\mu^{2}}\|e^{-sH}f\|_{L^{2}(\R)}^{2}.
    \end{aligned}
    \end{equation}
    It follows from Lemma~\ref{le:spectra} that 
    \begin{equation*}
    \begin{aligned}
        \|e^{-tH}\Pi_{\mu}f\|_{L^{2}(\R)}
        &\le Ce^{3\mu}\left(\|e^{-tH}f\|_{L^{2}(\Omega)}+
         \|e^{-tH}(1-\Pi_{\mu})f\|_{L^{2}(\R)}        
        \right)\\
        &\le Ce^{3\mu}\left( \|e^{-tH}f\|_{L^{2}(\Omega)}+e^{-(t-s)\mu^{2}}
         \|e^{-sH}f\|_{L^{2}(\R)}        \right).
        \end{aligned}
    \end{equation*}
    From AM-GM inequality, for any $\mu>0$, $\alpha\in (0,1)$ and $0\le s<t<\infty$, we obtain
    \begin{equation*}
        3\mu\le \alpha(t-s)\mu^{2}+\frac{3}{\alpha(t-s)}\Longrightarrow 
  3\mu- \alpha(t-s)\mu^{2}\le\frac{3}{\alpha(t-s)},
        \end{equation*}
        which implies
        \begin{equation}\label{est:etHOR}
    \begin{aligned}
        \|e^{-tH}\Pi_{\mu}f\|_{L^{2}(\R)}
        &\le Ce^{\frac{3}{\alpha(t-s)}  }      
    e^{\alpha(t-s)\mu^{2}}    
    \|e^{-tH}f\|_{L^{2}(\Omega)}\\
    &+ Ce^{\frac{3}{\alpha(t-s)}  }    e^{-(1-\alpha)(t-s)\mu^{2}}
\|e^{-sH}f\|_{L^{2}(\R)}.
        \end{aligned}
    \end{equation}
    Now we make the choice of $\mu \in(0,\infty)$ such that 
    \begin{equation*}
        e^{(t-s)\mu^{2}}=\frac{ \|e^{-sH}f\|_{L^{2}(\R)}}{ \|e^{-tH}f\|_{L^{2}(\Omega)}}\in (1,\infty).
    \end{equation*}
    This is always possible since for $f\ne 0$, we find
    \begin{equation*}
0<\|e^{-tH}f\|_{L^{2}(\Omega)}\le\|e^{-tH}f\|_{L^{2}(\R)}\le \|e^{-sH}f\|_{L^{2}(\R)}.
        \end{equation*}
        Therefore, from~\eqref{est:tH} and~\eqref{est:etHOR}, we conclude that 
        \begin{equation*}
        \|e^{-tH}f\|_{L^{2}(\R)}\le (2C+1)e^{\frac{3}{\alpha(t-s)}}\|e^{-tH}f\|_{L^{2}(\Omega)}^{1-\alpha}\|e^{-sH}f\|_{L^{2}(\R)}^{\alpha}.
    \end{equation*}
    The proof of Lemma~\ref{le:inter} is complete.
    \end{proof}

    We are in a position to complete the proof of Proposition~\ref{prop:heat}.
    \begin{proof}
    [End of the proof of Proposition~\ref{prop:heat}]
    We split the proof into the following two steps.

\smallskip
    \textbf{Step 1.} Recursive estimates.
    First, from Lemma~\ref{le:inter}, for $0<t_{1}<t_{2}\le T<\infty$, $t\in \left(\frac{t_{1}+t_{2}}{2},t_{2}\right)$ and $\alpha=\frac{1}{4}$, we have 
    \begin{equation*}
    \begin{aligned}
        \|u(t_{2})\|_{L^{2}(\R)}\le \|u(t)\|_{L^{2}(\R)}
        &\le Ce^{\frac{12}{t-t_{1}}}\|u(t)\|_{L^{2}(\Omega)}^{\frac{3}{4}}\|u(t_{1})\|_{L^{2}(\R)}^{\frac{1}{4}}\\
        &\le Ce^{\frac{24}{t_{2}-t_{1}}}\|u(t)\|_{L^{2}(\Omega)}^{\frac{3}{4}}\|u(t_{1})\|_{L^{2}(\R)}^{\frac{1}{4}},
        \end{aligned}
    \end{equation*}
    where $C=C(V,\zeta)>2$ is a constant depending only on $V$ and $\zeta$. 
    Integrating the above inequality over $\left(\frac{t_{1}+t_{2}}{2},t_{2}\right)$ and then using H\"older inequality, we see that 
    \begin{equation*}
        \|u(t_{2})\|_{L^{2}(\R)}^{2}
        \le C^{2}e^{\frac{48}{t_{2}-t_{1}}}\left(\frac{t_{2}-t_{1}}{2}\right)^{-\frac{3}{4}}\|u(t_{1})\|_{L^{2}(\R)}^{\frac{1}{2}}\left(\int_{t_{1}}^{t_{2}}\|u(t)\|_{L^{2}(\Omega)}^{2}\d t\right)^{\frac{3}{4}}.
    \end{equation*}
    It follows that 
    \begin{equation*}
    \begin{aligned}
        \|u(t_{2})\|_{L^{2}(\R)}^{2}e^{-\frac{99}{t_{2}-t_{1}}}
        \le& C^{3} \left({t_{2}-t_{1}}\right)^{-\frac{3}{4}}\left(\|u(t_{1})\|_{L^{2}(\R)}^{2}e^{-\frac{198}{t_{2}-t_{1}}} \right)^{\frac{1}{4}}   \\
        &\times 
        e^{-\frac{3}{2(t_{2}-t_{1})}}\left(\int_{t_{1}}^{t_{2}}\|u(t)\|_{L^{2}(\Omega)}^{2}\d t\right)^{\frac{3}{4}}\\
        \le &\left(\|u(t_{1})\|_{L^{2}(\R)}^{2}e^{-\frac{198}{t_{2}-t_{1}}} \right)^{\frac{1}{4}}\left(C^{4}(t_{2}-t_{1})
        \int_{t_{1}}^{t_{2}}\|u(t)\|_{L^{2}(\Omega)}^{2}\d t        
        \right)^{\frac{3}{4}}.
        \end{aligned}
        \end{equation*}
        Based on the above inequality and Young's inequality, we conclude that 
        \begin{equation}\label{est:u2u1}
 \|u(t_{2})\|_{L^{2}(\R)}^{2}e^{-\frac{99}{t_{2}-t_{1}}}
 \le \frac{1}{4}\|u(t_{1})\|_{L^{2}(\R)}^{2}e^{-\frac{198}{t_{2}-t_{1}}}
 +
 C^{4}(t_{2}-t_{1})
        \int_{t_{1}}^{t_{2}}\|u(t)\|_{L^{2}(\Omega)}^{2}\d t.        
     \end{equation}

    \smallskip
    \textbf{Step 2.} Conclusion. For $T>0$ and $n\in \mathbb{N}$, we set
    \begin{equation*}
        S_{n}=\frac{T}{2^{n}}\quad \mbox{and}\quad
        a_{n}=\left\|u(S_{n})\right\|_{L^{2}(\R)}^{2}e^{-\frac{99}{S_{n}-S_{n+1}}}.
        \end{equation*}
        By an elementary computation,
        \begin{equation*}
          S_{n+1}-S_{n+2} = \frac{S_{n}-S_{n+1}}{2}
          \Longrightarrow 
          a_{n+1}=\left\|u(S_{n+1})\right\|_{L^{2}(\R)}^{2}e^{-\frac{198}{S_{n}-S_{n+1}}}.
          \end{equation*}
          Therefore, from~\eqref{est:u2u1}, for any $n\in \mathbb{N}$, we have
          \begin{equation*}
              a_{n}\le \frac{1}{4}a_{n+1}+C^{4}T\int_{\frac{T}{2^{n+1}}}^{\frac{T}{2^{n}}}\|u(t)\|_{L^{2}(\Omega)}^{2}\d t,
          \end{equation*}
          which implies 
          \begin{equation*}
         \|u(T)\|_{L^{2}(\R)}^{2}e^{-\frac{198}{T}} =    a_{0}\le \frac{1}{4}a_{n}+C^{4}T\int_{0}^{T}\|u(t)\|_{L^{2}(\Omega)}^{2}\d t.
          \end{equation*}
          Using the fact that $a_{n}\to 0$ as $n\to\infty$, we complete the proof of Proposition~\ref{prop:heat}.
    \end{proof}

\end{document}